\documentclass[a4paper,11pt,reqno]{amsart}
\usepackage[english]{babel}
\usepackage[T1]{fontenc}

\usepackage{amsmath,amsthm,amssymb}
\usepackage{tikz}
\usepackage{bbm%,fontspec
}

\usepackage{shuffle}

\newtheorem{theo}{Theorem}[section]
\newtheorem{lemm}[theo]{Lemma}
\newtheorem{prop}[theo]{Proposition}
\newtheorem{coro}[theo]{Corollary}

\theoremstyle{definition}
\newtheorem{exam}[theo]{Example}
\newtheorem{defi}[theo]{Definition}
\newtheorem{rema}[theo]{Remark}

\title[Dual braid monoid algebras]{Koszulity of dual braid monoid algebras via cluster complexes}

\author{Matthieu Josuat-Verg\`es}
\address{Universit\'e de Paris, CNRS, IRIF (Institut de Recherche en Informatique Fondamentale, UMR8243)} 
\author{Philippe Nadeau}
\address{Universit\'e de Lyon, CNRS, Institut Camille Jordan,  UMR5208, 69622 Villeurbanne Cedex, France.}

\newcommand{\eps}{\epsilon}
\newcommand{\lra}{\longrightarrow}

\newcommand{\ba}{\mathbf{a}}
\newcommand{\bb}{\mathbf{b}}
\newcommand{\bc}{\mathbf{c}}
\newcommand{\bm}{\mathbf{m}}
\newcommand{\bs}{\mathbf{s}}
\newcommand{\bt}{\mathbf{t}}
\newcommand{\bu}{\mathbf{u}}
\newcommand{\bv}{\mathbf{v}}
\newcommand{\bw}{\mathbf{w}}
\newcommand{\bx}{\mathbf{x}}
\newcommand{\by}{\mathbf{y}}
\newcommand{\bT}{\mathbf{T}}
\newcommand{\bB}{\mathbf{B}}

\newcommand{\bNC}{\mathbf{NC}}

\newcommand{\bk}{k}

\newcommand{\ds}{\mathbbm{s}}
\newcommand{\dt}{\mathbbm{t}}
\newcommand{\du}{\mathbbm{u}}
\newcommand{\dv}{\mathbbm{v}}
\newcommand{\dx}{\mathbbm{x}}
\newcommand{\dT}{\mathbbm{T}}
\newcommand{\df}{\mathbbm{f}}

\newcommand\dif{\mathop{}\!\mathrm{d}}

\newcommand{\Dbm}{\mathbf{D}}
\newcommand{\Alg}{\mathcal{A}}
\newcommand{\Dual}{\mathcal{P}}

\newcommand{\tsh}{\mathrel{\tilde{\shuffle}}}

\DeclareMathOperator{\nc}{nc}
\DeclareMathOperator{\bnc}{\mathbf{nc}}

\DeclareMathOperator{\TT}{T}

\DeclareMathOperator{\tr}{tr}
\DeclareMathOperator{\im}{im}

\newcommand{\lt}{\ell_T}
\newcommand{\lb}{[\![}
\newcommand{\rb}{]\!]}

\begin{document}

%\setmathfont[range={\varnothing}]{XITS Math}

\begin{abstract}
The dual braid monoid was introduced by Bessis in his work on complex reflection arrangements. The goal of this work is to show that Koszul duality provides a nice interplay between the dual braid monoid and the cluster complex introduced by Fomin and Zelevinsky.  Firstly, we prove koszulity of the dual braid monoid algebra, by building explicitly the minimal free resolution of the ground field.  This is done explicitly using some chains complexes defined in terms of the positive part of the cluster complex.  Secondly, we derive various properties of the quadratic dual algebra.  We show that it is naturally graded by the noncrossing partition lattice.  We get an explicit basis, naturally indexed by positive faces of the cluster complex.  Moreover, we find the structure constants via a geometric rule in terms of the cluster fan.  Eventually, we realize this dual algebra as a quotient of a Nichols algebra.  This latter fact makes a connection with results of Zhang, who used the same algebra to compute the homology of Milnor fibers of reflection arrangements.
\end{abstract}

\maketitle

%%%%%%%%%%%%%%%%%%%%%%
\section{Introduction}
%%%%%%%%%%%%%%%%%%%%%%

\subsection{Context}

The {\it dual braid monoid} $\Dbm(W)$ of a finite Coxeter group $W$ (with respect to a Coxeter element $c$) was introduced by Bessis~\cite{Bes03}, as a set of positive elements inside the braid group $\bB(W)$.  In its more general form associated to a well-generated complex reflection group, this monoid turned out to be an important tool in Bessis' solution of the $K(\pi,1)$ problem for complex reflection arrangements~\cite{Bes15}.  This monoid has a rich structure: just like the braid group it can be defined algebraically or topologically, and it is a {\it Garside monoid}~\cite{DDGKM15}, a property that implies in particular the existence of canonical factorizations of each element.  A combinatorial byproduct is the definition of {\it generalized noncrossing partitions}, also known as {\it simple braids}.  They were introduced independently by Brady and Watt~\cite{BWKP1}, also in the context of the $K(\pi,1)$ problem for finite type Artin groups.  The same authors proved in~\cite{BWLattice} that the poset $NC(W)$ of noncrossing partitions is a lattice, a property that is an important ingredient of the Garside structure of $\Dbm(W)$.  The cardinality of $NC(W)$ is a generalized Catalan number, called {\it $W$-Catalan number}, see~\cite{Armstrong06} for a survey.  

In the context of cluster algebras and Zamolodchikov’s conjecture about $Y$-systems, Fomin and Zelevinsky~\cite{FZ03} introduced the {\it cluster complex} $\Delta(\Phi)$ of a finite type root system $\Phi$.  Its vertices can be identified with {\it cluster variables} of the cluster algebra of type $\Phi$, and its facets with {\it clusters} of the same cluster algebra.  This complex can be realized geometrically using a correspondence between cluster variables and {\it almost positive roots}:  each cluster correspond to a simplicial cone and together they form a complete simplicial fan, the {\it cluster fan}. Fomin and Zelevinsky~\cite{FZ03} also conjectured that this fan is the normal fan of a polytope.  These polytopes, called the {\it generalized associahedra}, were constructed by them in a joint work with Chapoton~\cite{CFZ05}.  The number of clusters (equivalently, the number of vertices in the generalized associahedron) turns out to be the $W$-Catalan number associated to the Weyl group $W$ of $\Phi$.  Following this observation, there has been an important combinatorial interplay between noncrossing partitions and clusters (see for example~\cite{Armstrong06,AthaEnum,AthaBrad06,ABW07,Cha05,McCammond06}).

A more specific connection is the link obtained in~\cite{AthaEnum,Cha05} between the characteristic polynomial of $NC(W)$ and the $f$-polynomial of $\Delta^+(W)$.  Here, $\Delta^+(W)$ is the {\it positive part} of $\Delta(W)$, the full subcomplex obtained by keeping only positive roots as its vertices (and we refer to the group $W$ rather than its root system).  It can be stated as follows:
\begin{equation} \label{eq:muNC}
 \sum_{F \in \Delta^+(W)} (-q)^{\dim(F)+1}
 =
 \sum_{ w \in NC(W) }  \mu_{NC}(w) q^{\lt(w)} 
\end{equation}
where $\mu_{NC}$ and $\lt$ are respectively the Möbius function and the rank function of $NC(W)$, see next section for details.   This is particularly relevant in the context of the dual braid monoid, as the {\it growth function} of $\Dbm(W)$ is the inverse of the polynomial in~\eqref{eq:muNC}. This is part of the Cartier--Foata theory~\cite{Car69}; see also Ishibe and Saito~\cite{IS17} for a recent exposition and related developments.

Explicitly, the growth function $\sum_{\bb\in\Dbm(W)} q^{\deg(b)}$ of $\Dbm(W)$ is given by
\begin{equation} \label{eq:growth}
 \sum_{\bb\in\Dbm(W)} q^{\deg(b)} = \Big( \sum_{\bb \in \Dbm } \mu_{\Dbm(W)} (\bb) q^{|\bb|} \Big)^{-1}
\end{equation}
where $\mu$ and $|.|$ are respectively the Möbius function and the rank function of $\Dbm(W)$, endowed with divisibility order.  Following Albenque and Nadeau~\cite{Alb09}, the Möbius function of $\Dbm(W)$ vanishes outside the set of simple braids, so that the right hand sides of \eqref{eq:muNC} and \eqref{eq:growth} are inverse of each other.  We refer to the next section for details.

\subsection{Outline of the results}

We will see through this work that there is an algebraic relation between the dual braid monoid and the positive part of the cluster complex, via the notion of {\it Koszul duality}~\cite{Fro99,Kra05,PP05}.  

Consider the monoid algebra $\bk[\Dbm(W)]$ over a ground field $\bk$.  Throughout, it will be denoted:
\[
  \Alg(W) := \bk[\Dbm(W)].
\]
It follows from the algebraic definition of $\Dbm(W)$ that $\Alg(W)$ is a {\it quadratic algebra}, and in particular a connected graded algebra.  This gives $\bk$ a structure of $\Alg(W)$-module via the {\it augmentation map} $\epsilon$, which by definition is the projection on the degree $0$ component.  Koszulity of $\Alg(W)$ is then characterized by a property of the minimal free resolution of $\bk$, namely all boundary maps have homogeneous degree 1.  

Our first goal is to show:

\begin{theo}  \label{theo:koszul_intro}
  $\Alg(W)$ is a Koszul algebra. 
\end{theo}

To describe our method, introduce for $-1 \leq i\leq n-1$ the free $\Alg(W)$-module with basis indexed by $\Delta^+_i(W)$, the set of $i$-dimensional faces in $\Delta^+(W)$.  This free module is denoted $\mathcal{C}_i $.  

\begin{prop}
  There exist explicit boundary maps $\partial_i : \mathcal{C}_i \lra \mathcal{C}_{i-1}$ such that the minimal free resolution of $\bk$ is
  \[
      0 \lra \mathcal{C}_{n-1} \stackrel{\partial_{n-1}}{\lra} \dots \stackrel{\partial_0}{\lra} \mathcal{C}_{-1} \stackrel{\epsilon}{\lra} \bk \lra 0.
  \]
\end{prop}

The nontrivial part consists in checking that the complex is exact.  This is done by seeing it as a direct sum of exact complexes, where each summand is the chain complex of a (topologically trivial) subcomplex of $\Delta^+(W)$.   As these maps are homogeneous of degree 1, we immediately get Theorem~\ref{theo:koszul_intro}.

Our second goal is to study the quadratic dual of $\Alg(W)$.  Throughout, it will be denoted
\[
  \Dual(W) := \Alg(W)^!.
\]
From the construction or the minimal resolution, it follows that the Hilbert series of this algebra is the $f$-polynomial of $\Delta^+(W)$:
\[
  \sum_{i=0}^{n} \dim( \Dual_i(W) ) q^i
  =
  \sum_{F \in \Delta^+(W)} q^{\dim(F)+1}.
\]  
  Equation~\eqref{eq:growth} can thus be interpreted as the relation between the Hilbert series of a Koszul algebra and that of its quadratic dual.

A presentation of $\Dual(W)$ is easily obtained in terms of the presentation of $\Alg(W)$ (see Theorem~\ref{theo:dual_presentation}).  We will also show that $\Dual(W)$ is naturally graded by $NC(W)$, and give the construction of an explicit basis indexed by $\Delta^+(W)$. 

In Section~\ref{sec:dualproduct}, we obtain a formula for the structure constants of $\mathcal{P}(W)$ with respect to the basis obtained in the preceding section.  This is a geometric rule that relies on the cluster fan.  More precisely, each face of the complex $\Delta^+(W)$ corresponds to a cone in this fan, and finding the expansion of a product in the algebra is given by such cones contained in some bigger cone.

In Section~\ref{sec:shuffle}, we introduce a {\it Nichols algebra} $\mathcal{N}(W)$ that is particularly relevant: it will be shown that $\Dual(W)$ is a quotient of $\mathcal{N}(W)$ (Theorem~\ref{theo_iso_nicholsquotient}), and some properties are more easily seen from this construction than from the presentation of $\Dual(W)$.  This point of view makes the link with Zhang's thesis~\cite{Zhang20} where the same algebra was introduced in order to compute the homology of Milnor fibers of reflection arrangements.  In particular, we explain how this give an alternative path to the koszulity of $\Alg(W)$ and $\Dual(W)$.

In Section~\ref{sec:cyclic_action}, we investigate the cyclic action generated by the Coxeter element on the algebras $\Alg(W)$ and $\Dual(W)$.  We get explicit formulas for the characters.  We get a new simple proof of a result of Zhang's thesis~\cite{Zhang20} that makes a connection with the homology of the noncrossing partition lattice. 

\subsection{Organization of the paper}

Sections~\ref{sec:dualbraidmonoid} and \ref{sec:quadratic} contains background material.  In the former, we recall preliminary notions related with finite Coxeter groups, noncrossing partitions, the dual braid monoid and its growth function.  In the latter, we recall the concept of Koszul duality and the role of the quadratic dual in this context.

In Section \ref{sec:koszulity}, we show that $\mathcal{A}(W)$ is a Koszul algebra (Corollary~\ref{coro:koszulity}), via the explicit construction of the minimal free resolution of ${\bk}$ (Theorem~\ref{theo:koszulity}).  This section relies on Appendix~\ref{sec:clustercomplex}, where we gather relevant material and bibliography about the cluster complex. 

In Sections~\ref{sec:propertiesdual}--\ref{sec:cyclic_action}, we give various properties of the dual algebra $\mathcal{P}(W)$, as outlined above.

Finally, in Section~\ref{sec:questions} we discuss possible extensions of this work to other kind of braid groups, beyond finite type Artin groups.   
\smallskip

{\bf Note:} While this manuscript was in the latter stages of redaction, the authors were made aware of the PhD thesis recently defended by Yang Zhang~\cite{Zhang20}. Quite remarkably, the algebras $\Dual(W)$ are also introduced in this work, albeit via a completely different path. In particular, it does not arise as the quadratic dual of $\mathcal{A}(W)$. We will mention the results of this work pertinent to ours in the course of the manuscript.

%%%%%%%%%%%%%%%%%%%%%%%%%%%%%%%
\section{The dual braid monoid}
\label{sec:dualbraidmonoid}
%%%%%%%%%%%%%%%%%%%%%%%%%%%%%%%

We review useful definitions and properties, and refer to~\cite{Humphreys90} for basic facts about finite Coxeter groups.  In the rest of this work, $(W,S)$ is a Coxeter system of rank $n$, which means that $W$ is finite and $S$ has cardinality $n$. The neutral element of $W$ is denoted $e$.  As is well-known, $W$ can be realized as a finite reflection group.

Let $T=\{\; wsw^{-1} \; | \; s\in  S, \; w\in W \}$ be the set of \textit{reflections} of $W$.  For $t \in T$, $w\in W$, let
\[
   t^w := w^{-1} t w \in T.  
\]
Note that $t^{w_1w_2} = (t^{w_1})^{w_2}$.  We  also fix a {\em standard Coxeter element $c$} in $W$, which by definition is the product of all simple reflections in some arbitrary order.  We can index $S=\{s_1,\dots,s_n\}$ so that $c=s_1\cdots s_n$.  There are various objects defined below that depend on $W$ and $c$.  In general, we omit the dependence in $c$.

Via the standard geometric representation, we see $W$ as a subgroup of the orthogonal group $O(\mathbb{R}^n)$.  For $t\in T$, we denote $\rho(t) \in \mathbb{R}^n$ the associated positive root (for a fixed choice of a generic positive half-space $\Pi \subset \mathbb{R}^n$).  

Each parabolic subgroup $P\subset W$ is seen as a reflection group in a linear subspace $V\subset \mathbb{R}^n$, where $V= \operatorname{Fix(P)}^{\perp} $ and $\operatorname{Fix(P)} = \{ v\in\mathbb{R}^n \;|\; \forall w\in P, \; w(v)=v \}$.  The half-space $\Pi\cap V$ endows $P$ with a natural choice of a set of positive roots, hence of a set of simple generators. 

\begin{lemm} \label{lem:rank2} 
Let $P\subset W$ be a rank 2 parabolic subgroup.  Its reflections can be indexed by $P \cap T=\{u_1,\dots,u_m\}$ in such a way that:
\begin{itemize}
    \item $u_{i+1} u_{i} = u_{i}u_{i-1}$ for $1\leq i \leq m$ (with $u_0=u_m$, i.e., indices are taken modulo $m$),
    \item the simple reflections of $P$ are $u_1$ and $u_m$.
\end{itemize}
\end{lemm}

We omit the proof. This lemma is particularly useful to deal with reflection orderings, see~\cite{Dye93}.  Note that reversing the order of the indexing of $P \cap T$ also gives a valid indexing, and there are only two of them.

\subsection{Noncrossing partitions}

Armstrong's work~\cite{Armstrong06} is a standard reference about this subject.  
For $w\in W$, the {\it absolute length} or {\it reflection length} of $w$ is:
\[
  \lt(w) := \min \big\{ k \geq 0 \; \big| \; \exists t_1,\dots,t_k \in T, \; t_1\cdots t_k=w \big\}.
\]
Since $T$ generates $W$, $\lt$ takes finite values.  A factorization $w=t_1\cdots t_k$ of $w\in W$ as a product of reflections is called {\it reduced} or {\it minimal} if $k=\lt(w)$.

\begin{lemm}[Carter~\cite{Car72}] \label{lin_indep}
Suppose that we have a reduced factorization $w=t_1\cdots t_k$ where $\lt(w)=k$, $t_i\in T$. 
Then $\rho(t_1), \dots, \rho(t_k)$ are linearly independent.
\end{lemm}

The \textit{absolute order} $\leq_T$ on $W$ is defined by $w\leq_T z$ if $\lt(w)+\lt(w^{-1}z)=\lt(z)$.   The order $\lt$ can also be characterized by the {\it subword property}: $w\leq_T z$ if and only if some reduced factorization of $w$ can be extracted as a subword of a reduced factorization of $z$ as a product of reflections.  See~\cite[Section~2.5]{Armstrong06}.

\begin{defi} \label{defi:NC}
We define the poset $NC(W)=NC(W,c)$ as the interval $[e,c]$ with respect to the partial order $\leq_T $.  This is a ranked poset with rank function $\lt$.  We denote $NC_j(W) \subset NC(W)$ the subset of elements of rank $j$. 
\end{defi}

Note that $\lt$ and $\leq_T$ are invariant under conjugation.  Because all Coxeter elements are conjugated, the isomorphism type of $NC(W)$ does not depend on $c$.

An important property is the following:

\begin{prop}[\cite{BWLattice}]
The poset $NC(W)$ is a lattice.
\end{prop}

Another point about noncrossing partitions is that they can be seen as parabolic Coxeter elements.

\begin{defi}
For $w \in W$, let $\Gamma(w)$ denote the smallest parabolic subgroup of $W$ containing $w$.
\end{defi}

It can be seen that the rank of $\Gamma(w)$ is $\lt(w)$.  Recall that each parabolic subgroup is endowed with a natural set of simple generators, see paragraph preceding Lemma~\ref{lem:rank2}. We also introduce the notation $\TT(w)=\Gamma(w)\cap T=\{t\in T~|~t\leq_T w\}$.  Then $\Gamma(w)$ is a reflection group with reflection set $\TT(w)$.

\begin{prop} \label{parabolic_standard_coxeter_element}
If $c$ is a standard Coxeter element, each $w\in NC(W)$ is a also a standard Coxeter element of $\Gamma(w)$.
\end{prop}

For a proof, see \cite[Proposition~3.1]{BJV18} and references therein.  In the crystallographic case, this can be obtained by representation theory, see~\cite{IngallsThomas}.

\subsection{The dual braid monoid}

Bessis~\cite{Bes03} defined the {\em dual braid monoid} $\Dbm(W)$ associated to $(W,S)$ and a Coxeter element $c$.  In fact, it is natural to define this in the context of the {\it dual Coxeter system} $(W,T)$, which means that we take all reflections as generators, rather than just simple reflections.  It is related with the {\it braid group} $\bB(W)$, by seeing $\Dbm(W)$ as the submonoid $\bB(W)$ of positive elements (products of generators, and no inverse of them).  Note that there is also a topological definition of this monoid given in \cite[Section~8]{Bes15}.

Let $\bT$ be a set in bijection with $T$, with the convention that for $t\in T$, $\bt$ is the corresponding element in $\bT$. Moreover, if $t,u\in T$, the element in $\bT$ corresponding to $t^u=utu\in T$ is denoted $\bt^\bu$.

\begin{defi} \label{defi:dbm}
The {\it dual braid monoid} $\Dbm(W)$ is defined by the presentation:
\[
  \Dbm(W) = \langle \; \bT \; | \; \bt\bu=\bu\bt^\bu \text{ if } tu\leq_T c \; \rangle
\]
As relations are homogeneous of degree 2, $\Dbm(W)$ has a natural grading, and we denote $|\bm|$ the length of $\bm$ as a product of generators.
\end{defi}

Because Coxeter elements are all conjugated, the isomorphism type of $\Dbm(W)$ as a homogeneous monoid does not depend on $c$.

There is another presentation of $\Dbm(W)$, using a bigger set of generators.  Just as $\bT$ is related to $T$, let us introduce a set $\bNC(W)$ in bijection with $NC(W)$ so that if $w\in NC(W)$, the corresponding element in $\bNC(W)$ is denoted $\bw$.

\begin{lemm}[{Bessis~\cite[Proposition~1.6.1]{Bes03}}]  \label{lemm:transitivity}
  Two minimal factorizations of $w = u_1 \cdots u_k = v_1 \cdots v_k $ of $w\in NC_k(W)$ can be connected by a finite sequence of {\emph Hurwicz moves}, which consists in replacing a factor $t_1 t_2$ with $t_2^{t_1} t_1$ or $t_2 t_1^{t_2}$ ($t_1,t_2 \in T$).
\end{lemm}

\begin{prop}
 A presentation of $\Dbm(W)$ is given by taking $\bNC(W)$ as a set of generators, with relations $\bv_1 \cdots \bv_j = \bw_1 \cdots \bw_k$ if:
 \begin{itemize}
     \item $v_1 \cdots v_j= w_1 \cdots w_k$,
     \item this element $v_1 \cdots v_j$ is in $NC(W)$,
     \item $ \lt(v_1 \cdots v_j) = \sum_{i=1}^j \lt(v_i) = 
     \sum_{i=1}^k \lt(w_i)$
 \end{itemize}
\end{prop}

\begin{proof}
By considering the case $j=k=2$ and $v_1,v_2,w_1,w_2 \in T$, we see that the generators and relations from Definition~\ref{defi:dbm} are included in those above.  It remains to see that we can add the new generators and relations without changing the structure.  

For $w\in NC_j(W)$, define $\bw = \bt_1 \cdots \bt_j$ where $t_1 \cdots t_j$ is a reduced factorization of $w$.  By Lemma~\ref{lemm:transitivity}, $\bw$ does not depend on the chosen reduced factorization.  We can add the new generator $\bw$ together with the relation $\bw = \bt_1 \cdots \bt_j$ without changing the structure.

Now let $v_1,\dots,v_j,w_1,\dots,w_k$ as above.  By considering a minimal factorization of each of these elements and again using Lemma~\ref{lemm:transitivity}, we obtain the relation $\bv_1 \cdots \bv_j = \bw_1 \cdots \bw_k$ as a consequence of the previously known relations.
\end{proof}

The elements of $\bNC(W)$ are called {\it simple braids}.  From the presentation of $\Dbm(W)$, we see that there is a well-defined monoid map $\Dbm(W) \to W$ defined by $\bw \mapsto w$ for $\bw \in \bNC(W)$.  (This map thus extends the natural bijection $\bNC(W) \to NC(W)$.)

\begin{exam}
\label{exam:typeA}
We will use the example of $W=\mathfrak{S}_n$ with the long cycle $c=(1,2,\cdots,n)$ as standard Coxeter element. In this case, $NC(W)$ is naturally identified as the set of noncrossing partitions of $\lb 1,n\rb$ ordered by inclusion \cite{Bia97}. Recall that a set partition is \emph{noncrossing} if no two blocks are crossing, where $B_1\neq B_2$ are crossing if there exist $i<j<k<l$  such that $i,k\in B_1$ and $j,l\in B_2$. 

The dual braid monoid $\Dbm(\mathfrak{S}_n)$ is the \emph{Birman-Ko-Lee monoid}~\cite{Bir98}, originally defined to answer the word and conjugacy problems in the Artin braid group. Its generators are $\bt_{i,j}$ for $1\leq i<j\leq n$, in bijection with the tranpositions $(i,j)$ in $S_n$. It is then defined by the congruences:
\[
\begin{cases}\bt_{i,j}\bt_{k,l}=\bt_{k,l}\bt_{i,j}\text{ for }i<j<k<l\text{ or }i<k<l<j;\\ 
\bt_{i,j}\bt_{j,k}=\bt_{j,k}\bt_{i,k}=\bt_{i,k}\bt_{i,j}\text{ for any }i<j<k.
\end{cases}\]
\end{exam}

\subsection{Garside monoids and Cartier-Foata theory}

We refer to \cite{DDGKM15} for Garside theory.

\begin{defi} \label{defgarside}
 Let $M$ be a monoid, and denote $x \prec_{\ell} y$ (resp., $ x \prec_r y$) if $x$ a left (resp., right) divisor of $y$.  We say that $M$ is a {\it Garside monoid} if:
 \begin{itemize}
     \item $M$ is {\it atomic} (i.e., each $x\in M$ has a finite number of left divisors and right divisors),
     \item  $M$ is {\it cancellative}  (i.e., for all $x,y,z\in M$ we have $
  xy = xz$ implies $y=z$, and $xz=yz$ implies $x=y$),
     \item $(M,\prec_{\ell})$ and $(M,\prec_r)$ are lattices, 
     \item %let $\delta$ denote the LCM of atoms of $M$, then
     there exists $\delta \in M$ (called a {\it Garside element}) such that 
     $\{x \in M\; : \; x \prec_{\ell} \delta \} = \{x \in M\; : \; x \prec_r \delta \}$, moreover this set is finite and generates $M$.
     \end{itemize}
\end{defi}

\begin{prop}[{Bessis~\cite[Theorem~2.3.2]{Bes03}}]
$\Dbm(W)$ is a Garside monoid, with $\bc=\bs_1\cdots \bs_n$ (the maximal simple braid associated to the Coxeter element $c = s_1\cdots s_n$) as a {\it Garside element}.  
\end{prop}

There is a poset isomorphism between $\bNC(W)$ (endowed with left or right divisibility) and $NC(W)$, thus the property of $NC(W)$ being a lattice mentioned above is crucially related with the Garside structure of $\Dbm(W)$.

Let $\mu(\bm):=\mu({\bf 1},\bm)$ be the M\"obius function of $\Dbm(W)$ as a poset under left divisibility, between the identity ${\bf 1}$ and any element $\bm\in \Dbm(W)$.  The work of Cartier-Foata~\cite{Car69} naturally applies to Garside monoids, and gives us the following identity:
\begin{align}
\label{eq:length_genseries_0}
\sum_{\bm\in\Dbm(W)}q^{|\bm|}&=\Big(\sum_{\bm\in\Dbm(W)}\mu(\bm)q^{|\bm|}\Big)^{-1}.
\end{align}

Work of Albenque and the second author~\cite[Theorem 2]{Alb09} gives an explicit expression for the M\"obius values $\mu(\bm)$ -- valid for a class of monoids extending Garside monoids --, from which it follows that $\mu(\bm)=0$ unless $\bm$ divides $\bc$.

Using the isomorphism between $NC(W)$ and $\bNC(W)$, we can rewrite \eqref{eq:length_genseries_0} as:
\begin{equation}
\label{eq:length_genseries}
\sum_{m\in\Dbm(W)}q^{|m|}=\Big(\sum_{w\in NC(W)}\mu(w)q^{\lt(w)}\Big)^{-1},
\end{equation}
where $\mu$ is here the M\"obius function of $NC(W)$.  In the case of Example~\ref{exam:typeA} for $n=4$, a computation of the M\"obius function gives the length generating function $(1-6q+10q^2-5q^3)^{-1}$ for the monoid $\Dbm(\mathfrak{S}_4)$.

For example, the right-hand side of~\eqref{eq:length_genseries} is the inverse of the left-hand side of~\eqref{eq:muNC}. 

%%%%%%%%%%%%%%%%%%%%%%%%%%%%%%%%%%%%%%%%%%
\section{Quadratic algebras and Koszul duality}
\label{sec:quadratic}
%%%%%%%%%%%%%%%%%%%%%%%%%%%%%%%%%%%%%%%%%%

Let $\bk$ be any commutative field, which will turn out to play no role in what follows.   Recall that for any graded $\bk$-algebra $A=\oplus_{i=0}^\infty A_n$ with finite-dimensional homogeneous components, its Hilbert series $\operatorname{Hilb}(A,q)$ is the formal power series defined by 
\[
  \operatorname{Hilb}(A,q)=\sum_{i=0}^\infty (\dim A_i) q^i. 
\]

We briefly recall the notion of Koszul algebra. For more details, we refer to the surveys~\cite{Fro99,UKra} or the book \cite{PP05} and references therein.

A graded algebra $Q$ is {\it quadratic} if it has a presentation $Q=\mathcal{T}(V)/\langle R\rangle$, where $V$ is finite dimensional vector space over $\bk$, $\mathcal{T}(V)=\oplus_{i\geq 0} V^{\otimes i}$ is its tensor algebra, and $R\subset V\otimes V$ is a $\bk$-subspace generating the ideal of relations $\langle R\rangle$.  Note that $V$ can be identified with $Q_1$, the degree 1 homogeneous component of $Q$.
 
Any quadratic algebra $Q$ possesses a {\it quadratic dual} $Q^!$, which is another quadratic $\bk$-algebra defined as follows. Write $Q= \mathcal{T}(V)/\langle R\rangle$ as above. Then, by definition $Q^{!}:=\mathcal{T}(V^*)/\langle R^\perp \rangle $ where $R^\perp \subset (V\otimes V)^*=V^*\otimes V^*$ is the space of linear forms on $V\otimes V$ which vanish on $R$.  Note that $V^*$ can be identified with $Q^!_1$, the degree 1 homogeneous component of $Q^!$.
 
%Concretely, finding the presentation of the quadratic dual can then simply done with linear algebra, using dual bases of $V$ and $V^*$. We will make explicit the presentation of $\Dual(W)$ in Section~\ref{sec:propertiesdual}.

We refer to \cite{BH93,PP05} for the notion of {\it graded free resolution} of a graded module.  Such resolutions always exist, and there is a {\it minimal} one which is unique up to isomorphism.  Koszul algebras can be characterized by a property of the minimal graded free resolution of $\bk$. Note that $\bk$ is naturally a $Q$-module, via the {\it augmentation map} $\epsilon:Q\to \bk$ defined by projection on the degree $0$ component $Q_0=\bk$.  A graded free resolution for this module has the form
\begin{equation}
\label{def_resolution}
   \cdots
   \stackrel{\partial_{3}}{\lra} Q^{c_3}
   \stackrel{\partial_{2}}{\lra} Q^{c_2}
   \stackrel{\partial_{1}}{\lra} Q^{c_1} \stackrel{\partial_0}{\lra} Q \stackrel{\epsilon}{\lra} \bk \lra 0.
\end{equation}
Minimality is characterized by the property $\partial_i(Q^{c_{i+1}}) \subset Q^+ \cdot Q^{c_i}$, i.e., the map $\partial_i$ has no component of homogeneous degree $0$.

\begin{defi}
The algebra $Q$ is \textit{Koszul} if each map $\partial_i$ in \eqref{def_resolution} is homogeneous of degree $1$.
\end{defi}

Here, we use the natural grading of each free $Q$-module coming from the grading of $Q$.  So, $\partial_i$ being homogeneous of degree $1$ means that the matrix of $\partial_i$ (with respect to canonical bases of the free modules) has coefficients in $Q_1$. 

Koszul algebras have a number of other characterizations, see~\cite{Fro99,UKra,PP05}. It is also known that $Q$ is Koszul if and only if $Q^!$ is.  When this is the case, the Hilbert series of $Q^!$ can be obtained either from the minimal resolution of $\bk$ by free $Q$-modules, or from the Hilbert series of $Q$:

\begin{prop}
\label{prop:numerical_koszul}
Let $Q$ be a quadratic algebra and $Q^!$ be its quadratic dual.  Suppose that $Q$ is a Koszul algebras, and the minimal resolution of $\bk$ by free $Q$-modules is as in~\eqref{def_resolution}.
Then, the Hilbert series of $Q^!$ is given by:
\begin{align} 
  \label{eq:numerical_koszul1}
  \operatorname{Hilb}(Q^!,q) &= 1+\sum_{i\geq 1} c_i q^i
\end{align}
Moreover, we have the relation:
\begin{align}
  \label{eq:numerical_koszul2}
  \operatorname{Hilb}(Q^!,q)   &= \operatorname{Hilb}(Q,-q)^{-1}.
\end{align}
\end{prop}

There is a more precise way to relate the minimal graded free resolution of $\bk$ with $Q^!$.  This will be explained in Section~\ref{sec:propertiesdual}.

Recall that $\Alg(W)$ was defined as the monoid algebra of $\bk[\Dbm(W)]$.  Therefore, following Definition~\ref{defi:dbm}, one has the quadratic presentation $\Alg(W) = \mathcal{T}(\bk^\bT) / \langle R\rangle $ where 
\begin{equation}
\label{eq:pres_alg}
R := \operatorname{Span}_{\bk} 
\big\{ 
\bt\otimes\bu - \bu\otimes \bt^\bu
\big\}
\subset \bk^\bT \otimes \bk^\bT.
\end{equation}

The Hilbert series of $\Alg(W)$ is thus the length generating series of $\Dbm(W)$, i.e., the left-hand side of~\eqref{eq:length_genseries}.  We get:
\begin{equation}
 \label{eq:hilbertprimal}
 \operatorname{Hilb}(\Alg(W),q)=\Big(\sum_{w\in NC(W)}\mu(w)q^{\lt(w)}\Big)^{-1}.
\end{equation}

\begin{defi}
We define $\Dual(W)$ as the quadratic dual of $\Alg(W)$; that is, $\Dual(W):=\Alg(W)^!$.
\end{defi}

%%%%%%%%%%%%%%%%%%%%%%%%%%%%%%%%%%%
\section{Koszulity of the dual braid monoid algebra}
\label{sec:koszulity}
%%%%%%%%%%%%%%%%%%%%%%%%%%%%%%%%%%%

In this section, we prove that $\Alg(W)$ is a Koszul algebra by building the minimal free resolution of the ground field $\bk$. This was previouly done in type $A$ and $B$ in \cite{Alb09}, using an {\it ad hoc} resolution in each case, which was built as a subcomplex of a bigger, non minimal  resolution. Here, we construct directly a minimal resolution for any $W$ and standard Coxeter element $c$, based on the positive cluster complex attached to this data.

As explained in the introduction, the idea is to build this resolution as a direct sum of complexes that are known to be exact.  A somewhat similar construction has been given by Kobayashi~\cite{Koba90} in the case of trace monoids.\smallskip

Our construction relies on the simplicial complex $\Delta^+(W)$,  the positive part of the cluster complex, and the related notion of $c$-compatible reflection ordering.  Their definitions and properties are given in Appendix~\ref{sec:clustercomplex}. We thus fix such a $c$-compatible reflection ordering $\prec$. We write $\{t_1\succ \dots \succ t_k \}$ to express that $\{t_1, \dots ,  t_k \}$ is indexed so that $t_1\succ \dots \succ t_k$.  Though Definition~\ref{def:clust1} show that reflection orderings are not strictly necessary, it will be convenient to have a canonical order on the vertices on $\Delta^+(W)$ in the following contexts:
\begin{itemize}
   \item when defining the reduced homology of the complex, each face $f\in\Delta^+(W)$ gets a canonical order, and geometrically this defines an orientation of the associated simplex,
   \item when we have an algebra having $T$ as a generating set, the total order can be used to obtain some bases as in the Poincaré-Birkhoff-Witt theorem, see~\cite{LV12}.
\end{itemize}

{\em Convention}: We consider that $W$ and $c$ are fixed. We write $\Alg,\Dual,NC_j,\dots$ instead of $\Alg(W),\Dual(W),NC_j(W),\dots$ to simplify the notation.\smallskip

\subsection{Construction of the minimal resolution}
\label{sub:minimal_resolution}

We use standard notions related with simplicial homology, see~\cite{Hat03,Koz08}.

\begin{defi}
For any set $X$, let $\bk^X$ denote the $\bk$-vector space freely spanned by $X$. Let $w\in NC_j$ with $1\leq j\leq n$.  The {\it augmented simplicial chain complex} of $\Delta^+(w)$ is
\begin{equation} 
\label{eq:complex}
  0 \lra \bk^{\Delta^+_{j-1}(w)} \stackrel{\beta_{j-1}}{\lra} \dots \stackrel{\beta_1}{\lra} \bk^{\Delta^+_0(w)} 
  \stackrel{\beta_0}{\lra} \bk^{\Delta^+_{-1}(w)} \lra 0
\end{equation}
where the {\it boundary maps} $\beta_i$ are defined on the basis as follows: if $f=\{t_0 \succ \dots \succ t_i\}\in\Delta^+_{i}(w)$, we have:
\[
  \beta_i(f) := \sum_{\ell=0}^i  (-1)^\ell  \cdot (f\backslash\{t_\ell\}) \in \bk^{\Delta^+_{i-1}(w)}.
\]
\end{defi}

\begin{prop}
For any $w\in NC_j$ with $1 \leq j\leq n$, the complex in \eqref{eq:complex} is exact.
\end{prop}

\begin{proof}
From Proposition~\ref{prop:deltaplusw}, it follows that the reduced simplicial homology (with coefficients in $\bk$) of $\Delta^+(w)$ is zero-dimensional in every degrees.  This means that the augmented simplicial chain complex is exact.
\end{proof}

\begin{defi}
 For $0\leq j \leq n-1$, the $\Alg$-module maps
 \[
   \partial_j : \Alg \otimes \bk^{\Delta^+_{j}} \lra \Alg \otimes \bk^{\Delta^+_{j-1}}
 \]
 are defined by, if $\bb\in\Dbm$ and $f = \{ t_0 \succ \dots \succ t_j \}\in\Delta^+_j$:
 \begin{equation} \label{def_partial}
  \partial_j( \bb \otimes f ) := \sum_{i=0}^j  (-1)^i  \cdot ( \bb \cdot \bt_i^{   \bt_{i-1}, \dots, \bt_0}) \otimes (f\backslash\{t_i\})
 \end{equation}
 where (recalling that $\bt^\bu\in\bT$ corresponds to $utu\in T$) the element $\bt_i^{   \bt_{i-1}, \dots, \bt_0} \in \bT$ corresponds to $t_0\cdots t_{i-1} t_i t_{i-1} \cdots t_0 \in T$.
\end{defi}

It is clear from the definition that $\partial_j$ is left $\Alg$-linear, for the natural structure of left $\Alg$-module on tensor products $\Alg \otimes \bk^ X$.

As $\Delta^+_{-1}=\{\varnothing\}$, we can identify $\Alg \otimes \bk^{\Delta^+_{-1}}$ with $\Alg \otimes \bk = \Alg$ and consider the augmentation $\epsilon$ as a map $\Alg \otimes \bk^{\Delta^+_{-1}} \to \bk$.

\begin{theo}
\label{theo:koszulity}
The diagram
\begin{equation}  \label{def_complex}
  0 \lra \Alg \otimes \bk^{\Delta^+_{n-1}} \stackrel{\partial_{n-1}}{\lra} \dots \stackrel{\partial_0}{\lra} \Alg \otimes \bk^{\Delta^+_{-1}} \stackrel{\epsilon}{\lra} \bk \lra 0
\end{equation}
is the minimal free resolution of $\bk$ by $\Alg$-modules.
\end{theo}

The rest of this section is devoted to the proof of Theorem~\ref{theo:koszulity}. Before embarking on it, let us note that the maps $\partial_i$ are homogeneous of degree 1 by their definition~\eqref{def_partial}, and thus \eqref{def_complex} is a linear resolution. It follows then 

\begin{coro}[Theorem~\ref{theo:koszul_intro}]
\label{coro:koszulity}
  $\Alg$ and its quadratic dual $\Dual$ are Koszul algebras.
\end{coro}

\begin{defi}
For $\bb\in\Dbm$ and $-1\leq i \leq n-1$, we define:
\begin{equation}
   \Theta_i(\bb) :=  \operatorname{Span}_{\bk} \big\{ \ba\otimes f \in \Alg \otimes \bk^{\Delta^+_i} \; \big| \;  \ba \cdot \bnc(f)= \bb \big\}.
\end{equation}
\end{defi}

\begin{prop}
 For any $\bb \in \Dbm$ and $0\leq j \leq n-1$, we have:
\begin{equation} \label{stability_Theta}
  \partial_j(\Theta_j(\bb))\subset \Theta_{j-1}(\bb).
\end{equation}
 Moreover, let $w\in NC$ such that $\bw \in \bNC$ is the greatest common right divisor of $\bb$ and $\bc$ in $\Dbm$.  If $\bb\neq 1$, then $w\neq 1$  and the complex 
 \begin{equation} \label{eq:complex2}
  0 \lra \Theta_{n-1}(\bb) \stackrel{\partial_{n-1}}{\lra} \dots \stackrel{\partial_0}{\lra} 
  \Theta_{-1}(\bb) \lra 0
 \end{equation}
 is ($\bk$-linearly) isomorphic to the augmented simplicial chain complex of $\Delta^+(w)$ in~\eqref{eq:complex}.  In particular this complex is exact.
\end{prop}

\begin{proof}
 Let $\ba \otimes f \in \Theta_{j}(\bb)$. We denote $f
 = \{t_0 \succ \dots \succ t_j\} \in \Delta_j^+$.  We need to check that each term in the right hand side of~\eqref{def_partial} is in $\Theta_{j-1}(\bb)$.  Note that
 \[
    \nc( f\backslash \{t_i\} ) = t_0 \cdots t_{i-1} t_{i+1} \cdots t_j,
 \]
 so 
 \[
   (t_0\cdots t_{i-1} t_i t_{i-1} \cdots t_0) \cdot \, \nc(f\backslash \{t_i\} ) = t_0 \cdots t_j = \nc(f).
 \]
 Moreover this is a reduced factorization since $\lt(nc(f\backslash \{t_i\} )) = j-1 = \lt(nc(f))-1$.  We thus have the following relation in $\Dbm$:
 \[
   \bt_i^{   \bt_{i-1}, \dots, \bt_0} \cdot  \bnc(f\backslash \{t_i\}) = \bnc(f).
 \]
 So,
 \[
   \ba \cdot \bt_i^{   \bt_{i-1}, \dots, \bt_0} \cdot  \bnc(f\backslash \{t_i\}) = \ba\cdot \bnc(f) = \bb,
 \]
 the last equality coming from $\ba \otimes f \in \Theta_{j}(\bb)$.  It follows that the right hand side of~\eqref{def_partial} is in $\Theta_{j-1}(\bb)$.
  
 To show the second part of the proposition, consider the $\bk$-linear projections $\pi_j : \Theta_j(\bb) \to \bk^{\Delta^+_j}$, for $-1\leq j \leq n-1$, defined on the basis by:
 \[
   \pi_j(\ba\otimes f) := f.
 \]
 Since $\ba\cdot\bnc(f)=\bb$ and $\Dbm$ is a cancellable monoid, we can recover $\ba$ from $\bb$ and $f$.  So $\pi_j$ is injective.  
 
 Let us show that its image is:
 \begin{equation} \label{imageofpi}
   \operatorname{im}(\pi_j) = \operatorname{Span}_{\bk} \big\{ f\in \Delta_j^+ \; \big| \; \nc(f) \leq_T w \big\} = \bk^{\Delta^+_j(w)}
 \end{equation}
 where $w$ is as in the proposition.  From $\ba\cdot\bnc(f)=\bb$ if $\ba \otimes f \in \Theta_{j}(\bb)$, we get that $\bnc(f)$ is a right divisor of $\bb$.  It is also a right divisor of $\bc$ by definition, and of $\bw$ (which was defined as the greatest right common divisor of $\bb$ and $\bc$).  This shows the left-to-right inclusion in~\eqref{imageofpi}.  Reciprocally, let $f \in \Delta_j^+$ be such that $nc(f) \leq_T w$.  So $\bnc(f)$ is a right divisor of $\bw$, and of $\bb$. Let $\ba\in \Dbm$ be such that $\ba\cdot\bnc(f)=\bb$.  We thus have $\ba\otimes f \in \Theta_j(\bb)$ and its image by $\pi_j$ is $f$.
  
 The next property is
 \[
   \beta_j\circ \pi_j = \pi_{j-1} \circ \partial_j
 \]
 for $0\leq j \leq n-1$.  This is straightforward by applying $\pi_{j-1}$ on both sides of~\eqref{def_partial}.  It follows that the map $\pi_j$ realizes an isomorphism between the complexes in~\eqref{eq:complex} and~\eqref{eq:complex2}.
\end{proof}

Recall that we have an indentification $\Alg \otimes \bk^{\Delta^+_{-1}} \simeq \Alg$ and that 
\[
  \Alg^+ = \bigoplus_{i\geq 1} \Alg_i.
\]
Note that we have $\operatorname{im}(\partial_0) \subset \Alg^+ \otimes \bk^{\Delta^+_{-1}}$.

\begin{prop} \label{prop:complexwithoutk}
The diagram
\begin{align}  \label{def_complex_C}
  \begin{split}
  0 \lra \Alg \otimes \bk^{\Delta^+_{n-1}} \stackrel{\partial_{n-1}}{\lra} \dots \stackrel{\partial_1}{\lra}
  \Alg \otimes \bk^{\Delta^+_{0}}
  \stackrel{\partial_0}{\lra}  \\
  \Alg^+ \otimes \bk^{\Delta^+_{-1}}  \lra 0
  \end{split}
\end{align}
is an exact complex of $\Alg$-modules.
\end{prop}

\begin{proof}
As the maps are clearly $\Alg$-linear, it suffices to show that this is an exact complex of $\bk$-vector spaces.  

First, note that $\dim \Theta_i(\mathbf{1}) = 0$ if $i\geq0$, and $\Theta_{-1}(\mathbf{1})$ is 1-dimensional, generated by $\mathbf{1}\otimes \varnothing$.  Indeed, $\ba\cdot \bnc(f) = \mathbf{1}$ implies $\ba=\mathbf{1}$ and $\bnc(f)=\mathbf{1}$.

It follows that:
\[
  \bigoplus_{\bb\in\Dbm \backslash\{\mathbf{1}\}} \Theta_i(\bb) 
  = 
  \begin{cases}
    \Alg\otimes\bk^{\Delta^+_i} & \text{ if } i\geq0, \\
    \Alg^+\otimes\bk^{\Delta^+_{-1}} & \text{ if } i=-1,
  \end{cases}
\]
because the canonical bases of summands in the left-hand side form a partition of the canonical basis of the right-hand side.  

We can thus see the complex in~\eqref{def_complex_C} as the direct sum of the exact complexes in~\eqref{eq:complex2}, over $\bb \in \Dbm$ such that $ \bb \neq \mathbf{1}$.  The result follows.
\end{proof}

\begin{proof}[Proof of Theorem~\ref{theo:koszulity}]
As $\ker(\eps)=\Alg^+ \simeq \Alg^+ \otimes\bk^{\Delta^+_{-1}}$, we deduce from Proposition~\ref{prop:complexwithoutk} that the complex of $\Alg$-modules in \eqref{def_complex} is exact.  We have thus built a free resolution of $\bk$. 
\end{proof}

 The resolution in Theorem~\ref{theo:koszulity} has finite length.  In the context of Noetherian algebras, this property is useful: a consequence is finiteness of the global dimension, which permits to define the homological quadratic form.  However, it is easily seen that our algebra $\Alg$ is not Noetherian.  We refer to~\cite{UKra} for more on this subject.

\begin{coro}
\label{coro:hilbertdual}
$\Dual$ is a finite-dimensional algebra with Hilbert polynomial
\[
  \operatorname{Hilb}(\Dual,q)=\sum_{f \in \Delta^+} q^{\# f}.
\]
\end{coro}

\begin{proof}
This is the link between the minimal free resolution and the dual algebra, see the first equality in Proposition~\ref{prop:numerical_koszul}.
\end{proof}

This suggests that there exists a linear basis of $\Dual$ indexed by $\Delta^+$.  An explicit construction will be given in the next section.

It is interesting to note that we get a new proof of the combinatorial identity in~\eqref{eq:muNC}.  Equation~\eqref{eq:hilbertprimal}, and the previous corollary, give the respective Hilbert series of $\Alg$ and $\Dual$.  So, the combinatorial identity follows from the second equality in Proposition~\ref{prop:numerical_koszul}.

%%%%%%%%%%%%%%%%%%%%%%%%%%%%%%%%%
\section{Properties of the dual algebra}
\label{sec:propertiesdual}
%%%%%%%%%%%%%%%%%%%%%%%%%%%%%%%%%
From the previous section, we know that $\Dual$ is a Koszul algebra with a nice Hilbert function.  Our goal is now to get further properties, and deepen the link between the algebra $\Dual$ and the complex $\Delta^+$.  It is straightforward to describe a presentation of $\Dual$, but it is not always sufficient to get properties of the algebra.  Thus, we again use the minimal free resolution of $\bk$ from the previous section.

{\em As in the previous section, we drop the dependence on $W,c$ in the objects defined. We also fix a $c$-compatible reflection ordering $\prec$ on the set $T$ of reflections of $W$.}

%%%%%%%%%%%%%%%%%%%%%%%%%%%%%%%%%
\subsection{A presentation of the dual algebra}
\label{sec:presentation}
%%%%%%%%%%%%%%%%%%%%%%%%%%%%%%%%%

The algebra $\Dual$ has a generating set in bijection with $T$ and $\bf T$. To avoid confusion we denote it $\mathbb{T}$. Its elements corresponding to reflections $t$, $u$, etc., are denoted $\dt$, $\du$, etc.  By the discussion in Section~\ref{sec:quadratic}, $\mathbb{T}$ is the basis of $\Alg^!_1 \simeq \Dual_1$ dual to the basis $\bT$ of $\Alg_1$.

Recall that, by definition of $\Dbm$, we have $\Alg = \mathcal{T}(\bk^{\bT}) / \langle R \rangle$ where $R\subset \bk^{\bT} \otimes \bk^{\bT}$ is generated by $\bt\otimes \bu - \bu\otimes \bt^{\bu}$ for $t,u\in T$ such that $tu\leq_T c$.

\begin{theo}
\label{theo:dual_presentation}
We have $\Dual = \mathcal{T}(\bk^{\dT}) / \langle R^\perp\rangle$, where $R^\perp \subset \bk^{\dT}\otimes \bk^{\dT}$ is spanned by the elements:
\begin{enumerate}
\item \label{it:rel1} $\dt \otimes \dt$ for $t\in T$;
\item \label{it:rel2} $\dt \otimes \du$ for $t,u\in T$ such that $tu\not\leq_T c$;
\item \label{it:rel3} $\du_1\otimes\du_m+\du_m\otimes\du_{m-1}+\cdots+\du_2\otimes\du_1$ 
for each $w\in NC_2$, where $\TT(w) = \{u_1 \prec \dots \prec u_m\}$.
\end{enumerate}
\end{theo}

\begin{proof} For $i\in\{1,2,3\}$, let $R_i \subset \bk^{\dT}\otimes \bk^{\dT}$ be the subspace spanned by elements in $(i)$ as in the theorem.  Moreover, for $w\in NC_2$, let $R_w \subset \bk^{\dT}\otimes \bk^{\dT}$ be the subspace spanned by elements $\dt\otimes \du$ such that $tu=w$.  Its linear dual can be seen as the subspace $R^*_w \subset \bk^{\bT}\otimes \bk^{\bT}$ generated by elements $\bt\otimes \bu$ such that $tu=w$.
 
 Clearly, $\bk^{\dT}\otimes \bk^{\dT}$ is the direct sum of $R_1$, $R_2$, and $R_w$ for $w\in NC_2$.  
 
 By homogeneity of the relations defining $\Dbm$, we have $R = \oplus_{w\in NC_2} (R\cap R_w^*)$. So, $R^\perp$ is the direct sum of $R_1$, $R_2$, and $( R \cap R_w^* )^\perp$ for $w\in NC_2$ (where the orthogonal is taken inside $R_w$).  
 
 It remains to identify these subspaces $( R \cap R_w^* )^\perp$. 
 Let $u_1,\dots,u_m$ be as in the theorem.  It is easily seen that $R \cap R_w^*$ is generated by the elements $\bt_i\otimes \bt_{i-1} - \bt_j\otimes \bt_{j-1}$ (taking indices modulo $m$), so that the orthogonal in $R_w$ is $1$-dimensional generated by $\sum_{i=1}^m \dt_i\otimes\dt_{i-1}$.
\end{proof}

\begin{exam}
\label{exam:typeA2} Take $W=\mathfrak{S}_n$ and $c=(1,2,\dots,n)$ as in Example~\ref{exam:typeA}, where the relations for the dual braid monoid $\Dbm$ are given. The relations for the dual algebra are:
\[
 \begin{cases}
  \dt_{i,j}^2 &\text{ for all }i<j ;\\
  \dt_{i,j}\dt_{k,l} &\text{ if }i\leq  k<j\leq l\text{ or }k<i\leq l<j;\\
  \dt_{i,j}\dt_{k,l}+\dt_{k,l}\dt_{i,j} &\text{ if }i<j<k<l\text{ or }i<k<l<j;\\
  \dt_{i,j}\dt_{j,k}+\dt_{j,k}\dt_{i,k}+\dt_{i,k}\dt_{i,j} &\text{ for }i<j<k.
 \end{cases}
\]
\end{exam}

\begin{rema}
 Each term $\dt\otimes\du$ for $\dt,\du\in\dT$ appears exactly once in the relations given in Theorem~\ref{theo:dual_presentation}.  It follows that $\tau := \sum_{\dt\in\dT} \dt \in \Dual_1$ satisfies $\tau^2=0$.  In this situation, the maps $\Dual_i \to \Dual_{i+1}$ defined by left multiplication by $\tau$ form a complex of vector spaces.
\end{rema}

%%%%%%%%%%%%%%%%%%%%%%%%%%%%%%%%%
\subsection{Relating the minimal resolution to the dual algebra}
\label{sec:relating}
%%%%%%%%%%%%%%%%%%%%%%%%%%%%%%%%%

It is well known to experts that the maps $\partial_i$ from Section~\ref{sec:koszulity} are related with the product of $\Dual$.  For later use, we explain this point.

Let $\Dual^*$ denote the linear dual of $\Dual$. 
It is naturally a graded coalgebra, with the coproduct $\delta$ obtained by dualizing the product of $\Dual$.  In degree $1$, we have $\Dual_1^* = (\bk^{\dT})^* = \bk^{\bT}$.  So it is natural to denote generic elements of $\Dual^*$ with $\bx,\by$, etc.  

This coalgebra $\Dual^*$ can be used to define the {\it Koszul complex} of $\Alg$.  We denote 
\[
   \delta_{(1,j-1)} : \Dual^*_j \lra \Dual^*_1 \otimes \Dual^*_{j-1}
\]
the homogeneous part of $\delta$ of degree $(1,j-1)$. Using Sweedler's notation, we write $\delta_{(1,j-1)}(\bx) = \sum_{\delta,1,j-1} \bx' \otimes \bx''$.  We define $\Alg$-module maps $\eth_j : \Alg \otimes \Dual^*_j \lra \Alg \otimes \Dual^*_{j-1}$ by
\begin{equation} \label{rel_eth_delta}
  \eth_j( \bb\otimes \bx ) = \sum_{\delta,1,j-1} (\bb \bx') \otimes \bx''.
\end{equation}
From the general theory of Koszul algebras (see~\cite[Definition-Theorem~1, $xxi$]{Fro99} or \cite[Chapter~2]{PP05}), the minimal free resolution of $\bk$ is:
\begin{equation} \label{other_free_resolution}
  0 \lra \Alg \otimes \Dual^*_n \stackrel{\eth_{n}}{\lra} \dots
  \stackrel{\eth_{1}}{\lra}
  \Alg \otimes \Dual^*_0
  \stackrel{\eps}{\lra}
  \bk \lra 0.
\end{equation}
By uniqueness of this resolution, the complexes in~\eqref{def_complex} and~\eqref{other_free_resolution} are isomorphic.  This means there exist $\Alg$-module isomorphisms 
\[
  \Psi_i : 
  \Alg\otimes\bk^{\Delta^+_{i-1}} \to \Alg \otimes \Dual^*_i
\]
such that:
\begin{itemize}
 \item $\Psi_i$ is homogeneous of degree 0, for $0\leq i \leq n$, i.e.,~$\Psi_i = I \otimes \psi_i$ where $\psi_i: \bk^{\Delta^+_{i-1}}\to \Dual^*_i$ is a $\bk$-linear isomorphism,
 \item $\eth_i \circ \Psi_i = \Psi_{i-1} \circ \partial_{i-1}$, for $1\leq i \leq n$. 
\end{itemize}
The homogeneity of $\Psi_i$ is natural when working in the category of graded modules, see~\cite[Chapter 1.5]{BH93} for instance.  It is well-known that this kind of resolution can be built inductively: knowing $\Psi_{i-1}$, there is a construction of $\Psi_i$.  Consequently, $\psi_1$ can be chosen to be the natural identification $\bk^T \to \bk^{\bT}$, as can be checked by comparing $\eth_1$ and $\partial_0$.

The following comes as no surprise:

\begin{prop} \label{iso_coproduct}
  Define a linear isomorphism $\bk^{\Delta^+} \to \Dual^*$ by $\psi := \oplus_{i=0}^n \psi_i$.  Define $\eta := (\psi\otimes\psi)^{-1}\circ \delta \circ \psi$.
  Then $\eta$ is a graded coproduct on $\bk^{\Delta^+}$ such that:
  \begin{itemize}
   \item there is a coalgebra isomorphism $\psi: \bk^{\Delta^+} \to \Dual^*$, having the canonical isomorphism $\bk^T \to \bk^{\bT}$ as homogeneous component of degree $1$,
   \item Using Sweedler's notation as in~\eqref{rel_eth_delta}, we have:
   \begin{equation}  \label{rel_partial_eta}
     \partial_i(\bb\otimes f) = 
     \sum_{\eta,1,i-1} (\bb \mathbf{f}') \otimes f''.
   \end{equation}
  \end{itemize}
\end{prop}

\begin{proof}
 It is straightforward to check that $\eta$ is a graded coproduct on $\bk^{\Delta^+}$, and $\psi$ satisfies the condition in the first point of the proposition.
 
 Moreover, for $\bb\otimes f \in \Alg \otimes \bk^{\Delta^+_{i-1}}$, we have:
 \[
   \partial_{i-1}(\bb\otimes f) = (\Psi_{i-1}^{-1}\circ\eth_i\circ\Psi_i)(\bb\otimes f) = (\Psi_{i-1}^{-1} \circ \eth_i )(\bb\otimes \psi_i(f)).
 \]
 Using~\eqref{rel_eth_delta} and $\delta\circ\psi = (\psi\otimes\psi)\circ\eta$, this gives:
 \[
   \partial_{i-1}(\bb\otimes f) = \Psi_{i-1}^{-1}\Big( \sum_{\eta,1,i-1} \big(\bb \psi_1(f')) \otimes \psi_{i-1}(f'') \Big).
 \]
 From the assumption on $\psi_1$, and $\Psi_{i-1}=I \otimes \psi_{i-1}$, we get~\eqref{rel_partial_eta}.
\end{proof}

\begin{rema} \label{rem:unshuffle}
  One may ask if there is an explicit coproduct $\eta$ satisfying~\eqref{rel_partial_eta}.  The formula for $\partial_i$ suggests an unshuffling (the dual of a shuffle product), with a twist to take into account the conjugation (i.e., that we have $\bt_i^{   \bt_{i-1}, \dots, \bt_0}$ rather than just $\bt_i$.  This will lead us to the developments in Section~\ref{sec:shuffle}.
\end{rema}

%%%%%%%%%%%%%%%%%%%%%%%%%%%%%%%%%
\subsection{A grading of the dual algebra}
\label{sec:grading}
%%%%%%%%%%%%%%%%%%%%%%%%%%%%%%%%%
We show that $\Dual$ is a direct sum of subspaces indexed by noncrossing partitions, in such a way that the product behaves well with respect to this decomposition.

\begin{lemm}
 Let $\delta_{(1^j)}$ denote the homogeneous component of degree $(1,\dots,1)$ of the $j$-fold coproduct:
 \[
   \delta_{(1^j)} : \Dual^*_j \lra (\Dual^*_1)^{\otimes j}.
 \]
 Then the image of $\delta_{(1^j)}$ is spanned by elements of the form  $\bt_1 \otimes \dots \otimes \bt_j$ where $t_1\cdots t_j$ is a reduced factorization of some $w\in NC_j$.
\end{lemm}

\begin{proof}
 Note that $\psi_1^{\otimes j}$ is the natural identification between $(\bk^{T})^{\otimes j} \to (\bk^{\bT})^{\otimes j}$.  Using the isomorphism $\psi$ as in Proposition~\ref{iso_coproduct}, it suffices to prove the statement with $\bk^{\Delta^+}$ and the map $\eta_{(1^j)}$ defined in the same way as $\delta_{(1^j)}$.
  
 Let $f\in \Delta^+_{j-1}$.  By an induction on $j$, we show that $\eta_{(1^j)}(f)$ only contains terms $u_0\otimes \cdots \otimes u_{j-1}$ such that $u_0\cdots u_{j-1} = \nc(f)$.  Assume $j\geq 2$, since $j=1$ is immediate.  By \eqref{rel_partial_eta} and~\eqref{def_partial}, if $f=\{t_0 \succ \dots \succ t_{j-1}\} \in \Delta^+_{j-1}$, we have:
 \[
   \eta_{(1,j-1)} ( f ) =  \sum_{i=0}^{j-1}  (-1)^i  \cdot ( t_i^{ t_{i-1}, \dots, t_0}) \otimes (f\backslash\{t_i\}).
 \]
 Then, by coassociativity we have:
 \[
   \eta_{(1^j)} (f) = \sum_{i=0}^{j-1}  (-1)^i  \cdot ( t_i^{ t_{i-1}, \dots, t_0}) \otimes \eta_{(1^{j-1})} (f\backslash\{t_i\}).
 \]
 Using the induction hypothesis, this is a linear combination of $( t_i^{ t_{i-1}, \dots, t_0})\otimes u_0 \otimes \dots \otimes u_{j-2}$ where $u_0\cdots u_{j-2} = \nc(f\backslash\{t_i\})$.  We have seen previously that $( t_i^{ t_{i-1}, \dots, t_0}) \cdot \nc(f\backslash\{t_i\}) = \nc(f) $.  The result follows.
\end{proof}

\begin{prop} \label{prop:vanishingprop}
 Let $j\geq 0$ and $t_1,\dots,t_j\in T$.  If $t_1\cdots t_j \notin NC_j$, we have $\dt_1\cdots\dt_{j} = 0$ in $\Dual$.
\end{prop}

\begin{proof}
Assume that $t_1\cdots t_j \notin NC_j$.  We want to show that $\dt_1 \otimes \cdots \otimes \dt_j$ is in the kernel of the $j$-fold product $\Dual_1^{\otimes j} \to \Dual_j$.  By duality, this condition is equivalent to $\dt_1 \otimes \cdots \otimes \dt_j \in \operatorname{im} (\delta_{(1^j)})^{\perp}$.  By the previous lemma, $\operatorname{im}(\delta_{(1^j)})$ is linearly generated by elements that are orthogonal to $\dt_1 \otimes \cdots \otimes \dt_j$.  The result follows.
\end{proof}

\begin{rema}  \label{rem:vanishingprop}
  The results about $\Dual$ in the rest of this section will follow from the previous result.  One might ask if this proposition can be proved directly from the presentation of the algebra in Theorem~\ref{theo:dual_presentation} and rewriting techniques. We managed to do this in the simply-laced case. The general case is treated in Zhang's thesis~\cite{Zhang20}.
\end{rema}

\begin{defi} \label{def:PwFw}
 For $0\leq j\leq n$ and $w\in NC_j$, we define: 
 \begin{align}
  \Dual_w &:= 
  \operatorname{Span}_{\bk} \Big\{ \dt_1 \cdots \dt_j \; \Big| \; t_1\cdots t_j=w \Big\}
  \subset \Dual_j,
  \\
  F_w &:= \operatorname{Span}_{\bk} \Big\{ \dt_1\otimes \dots \otimes \dt_j \; \Big| \; t_1\cdots t_j=w \Big\} \subset \mathcal{T}(\bk^{\dT}). 
 \end{align}
\end{defi}

\begin{rema} \label{rema_rel3}
Keeping the notation above, note that the generators of $F_w$ contain no factor $\dt \otimes \dt$ and no factor $\dt \otimes \du$ where $tu\not\leq_T c$.  So, as vector spaces, $\Dual_w = F_w / (F_w \cap \langle R_3 \rangle )$ where $R_3$ is the span of Relations~\eqref{it:rel3} in Theorem~\ref{theo:dual_presentation}.
\end{rema}

\begin{prop} \label{prop:refinedgrading}
We have:
\begin{equation} \label{eq:refinedgrading}
 \Dual=\bigoplus_{w\in NC}  \Dual_w.
\end{equation}
This decomposition is graded in the following sense: for $w,w'\in NC$ and $(\dx,\dx') \in \Dual_w \times \Dual_{w'}$, we have:
\begin{itemize}
    \item $\dx\dx'\in \Dual_{ww'}$ if $ww'\in NC$ and $\lt(w) + \lt(w')=\lt(ww')$,
    \item $\dx\dx'=0$ otherwise.
\end{itemize}
\end{prop}

\begin{proof}
 Let us define $F = \bigoplus_{w\in NC} F_w$.  It follows from Proposition~\ref{prop:vanishingprop} that $\Dual$ a quotient of $F$ in a natural way.  By the argument in Remark~\ref{rema_rel3}, we have $\Dual = F / (F \cap \langle R_3 \rangle)$.  It remains to show that $F \cap \langle R_3 \rangle = \bigoplus_{w\in NC} ( F_w \cap \langle R_3 \rangle)$.  This easily follows from the fact that for each relation $\dt_1\otimes\dt_m + \dt_m\otimes\dt_{m-1} + \cdots + \dt_2\otimes\dt_1 \in R_3 $ as in Theorem~\ref{theo:dual_presentation}, we have $t_1 t_m = t_m t_{m-1} = \dots = t_3 t_2 = t_2 t_1$.

 The second part of the proposition follows from the definition of $\Dual_w$ and Proposition~\ref{prop:vanishingprop}.
\end{proof}

%%%%%%%%%%%%%%%%%%%%%%%%%%%%%%%%%
\subsection{A basis of the dual algebra}
\label{sec:basis}
%%%%%%%%%%%%%%%%%%%%%%%%%%%%%%%%%
We now construct a $\bk$-linear basis of $\Dual$, as announced after Corollary~\ref{coro:hilbertdual}. 

 If $f = \{ t_1\succ \dots \succ t_j \} \in \Delta^+$ is any face, we denote $\df := \dt_1 \cdots \dt_j \in \Dual_j$. Note that it is quite useful to have fixed an ordering of the vertices of $f$: since $\Dual$ is not commutative, one needs to fix the order in which the reflections $\dt_i$ are multiplied to get a well-defined element.

\begin{theo} \label{theo:basis}
A basis of $\Dual_w$ as a vector space is:
\begin{equation} \label{eq:basisw}
  \Big\{ \; \mathbbm{f} \; \Big| \; f \in \Delta^+,\; \nc(f)=w  \Big\}.
\end{equation}
In particular, a basis of $\Dual$ as a vector space is $\big\{ \df \; \big| \; f\in\Delta^+\big\}$.
\end{theo}

\begin{proof}
 Let $0\leq j \leq n$ and $w\in NC_j$, and consider a monomial $\dt_1 \cdots \dt_j$. Suppose that there exists $1\leq i<j$ such that $t_i\prec t_{i+1}$.  Let $\TT(t_it_{i+1}) =   \{u_1,\dots,u_m\}$, indexed so that $u_1\prec \cdots \prec u_m$.  We thus have the following relation in $\Dual$:
\begin{equation*}
  \du_1\du_m + \du_m\du_{m-1}+\cdots+\du_3\du_2+\du_2\du_1 = 0.
\end{equation*}
By the property of a reflection ordering, we have $t_i=u_1$ and $t_{i+1}=u_m$, as $u_1u_m = t_it_{i+1}$ is the unique increasing factorization of $t_it_{i+1}$.
We use the above relation to replace $\dt_i \dt_{i+1} = \du_1 \du_m $ with $-(\du_m\du_{m-1}+\cdots+\du_2\du_1)$ in $\dt_1 \cdots \dt_k$. The resulting monomials are all larger than $\dt_1 \cdots \dt_k$ for the lexicographic order on monomials induced by $\prec$.

It follows that any of the generating monomial of $\Dual_w$ can eventually be rewritten as a linear combination of decreasing monomials, i.e., elements of the set ~\eqref{eq:basisw}.  

It also follows that $\{\, \df \, : \, f\in\Delta^+\}$ spans $\Dual$.
By Corollary~\ref{coro:hilbertdual}, we know that $\dim \Dual = \# \Delta^+$, so this is a $\bk$-linear basis.  This permits us to conclude the proof.
\end{proof}

We will revisit this proof in the next section, and get an explicit expansion in Theorem~\ref{theo:prod}. Below, recall that $\mu_{NC}$ is the Möbius function of $NC$. 

\begin{coro} \label{coro:dim_mu}
 The dimensions of the homogeneous components of $\mathcal{P}$ are given by:
 \[
   \dim( \mathcal{P}_w ) = (-1)^{\lt(w)} \mu_{NC}(w).
 \]
\end{coro}

\begin{proof}
The number of $f \in \Delta^+$ such that $nc(f)=w$ is known to be the Möbius number $\mu_{NC}(1,w)$.  This follows from the case $w=c$.  See \cite{Cha05} for details.
\end{proof}

\begin{rema}  \label{rem:zhang}
  We have shown that each monomial in $\Dual(W)$ can be rewritten as a decreasing monomial, and these form a basis.  This means that we have built a {\it Poincaré-Birkhoff-Witt basis} (PBW basis) of $\Dual(W)$.  It is known that a quadratic algebra with a such a basis is Koszul (see~\cite[Chapter~4]{LV12}). 
  
  Therefore, we have an alternative path to prove koszulity of $\Alg(W)$ and $\Dual(W)$: if we can get the vanishing property (Proposition~\ref{prop:vanishingprop}) from the presentation of $\Dual(W)$ in Theorem~\ref{theo:dual_presentation}, we obtain Koszulity by means of the PBW basis.  As mentioned in Remark~\ref{rem:vanishingprop}, such a proof of Proposition~\ref{prop:vanishingprop} is given by Zhang~\cite{Zhang20}. 
\end{rema}

%%%%%%%%%%%%%%%%%%%%%%%%
\section{A geometric rule for the product in the dual algebra}
\label{sec:dualproduct}
%%%%%%%%%%%%%%%%%%%%%%%%

Having just built a basis of $\Dual$ in Theorem~\ref{theo:basis}, it is natural to ask about the {\em structure constants} of $\Dual$ with respect to this basis: that is, what is the expansion of the product of basis elements in the basis ?
The goal of this section is to give an elegant geometric rule to answer this question, see Theorem~\ref{theo:prod} and Corollary~\ref{coro:dual_product}.  In particular, we will see that the structure constants are in $\{-1,0,1\}$: that is, the desired expansion is {\em multiplicity-free}.\smallskip 

The result relies on a geometric realization of the cluster complex $\Delta$, that we describe now.

\subsection{The cluster fan}
\label{sec:clusterfan}

Fomin and Zelevinsky \cite{FZ03} showed that the cluster complex naturally defines a complete simplicial fan (see~\cite{Zie95} for terminology about simplicial fans and cones).  This result also holds in the present situation (where $W$ is possibly non crystallographic, and $c$ possibly non bipartite), see~\cite{RS09}.

\begin{defi}
If $t_1,\dots, t_j\in T$ are such that $t_1\cdots t_j \in NC_j$, we denote:
\[ 
  \gamma(t_1,\cdots,t_j) :=  \operatorname{Span}_{\mathbb{R}^+}  \{ \rho(t_1), \dots , \rho(t_j) \}.
\]
Moreover, if $f=\{t_1 \succ \dots \succ t_j\}\in \Delta^+$, we denote $\gamma(f):=\gamma(t_1,\dots,t_j)$.
\end{defi}

Note that with the above notation, $\gamma(t_1,\dots,t_j)$ and $\gamma(f)$ are simplicial cones of dimension $j$, as the generating vectors are linearly independent by Lemma~\ref{lin_indep}. Recall that a fan is a set of cones which such that is stable under taking faces, and intersection. The \emph{support} of a cone is the union of all its cones.   

\begin{prop}[\cite{FZ03,RS09,Cha05}]
  The cones $( \gamma(f) )_{f\in \Delta^+}$ form a simplicial fan. Its support is the cone generated by simple roots. 
\end{prop}

We refer to this fan as the \emph{positive cluster fan} (relative to $W,c$). This fan is clearly not complete, since it is defined as a subfan of the one corresponing to the whole cluster complex $\Delta$.  Rather than the total space, the union of $\gamma(f)$ for $f\in \Delta^+$ is the positive span of positive roots \cite{Cha05}. 

A similar result holds for each subcomplex $\Delta(w)$ with $w\in NC$ defined in Appendix~\ref{sec:clustercomplex}.  It follows from the previous proposition applied to the subgroup $\Gamma(w)$, with $w$ as its standard Coxeter element.

\begin{prop}  \label{prop:fans}
  For each $w\in NC$, the cones $( \gamma(f) )_{f \in \Delta^+(w)}$ form a simplicial fan.
\end{prop}

\begin{figure}[!ht]
    \centering
    \includegraphics{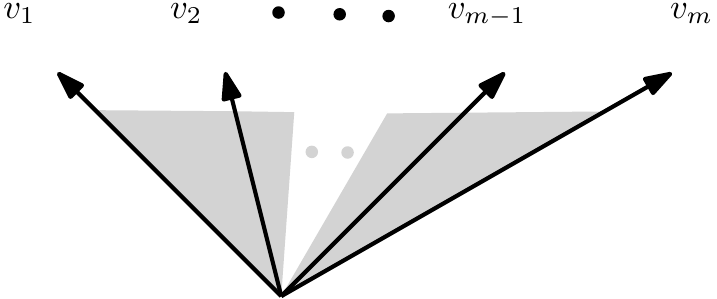}
    \caption{Positive cluster fan in rank $2$.}
    \label{fig:fan_dim_2}
\end{figure}

\subsection{Product in $\Dual$}

Through this section, we fix a tuple of reflections $(t_1,\dots,t_j)$, and let $w=t_1\cdots t_j$. By Proposition~\ref{prop:vanishingprop}, we know the product $\dt_1\cdots \dt_j$ is zero if $w \notin NC_j$. We can thus restrict to the following case: \emph{ $w \in NC_j$, and thus $(t_1,\dots,t_j)$ is a reduced expression for $w$}.
\smallskip
Let 
\[
  M(t_1,\dots,t_j) :=  \operatorname{Span}_{\mathbb{R}}  \{ \rho(t_1), \dots , \rho(t_j) \}. 
\]
This is the {\it moved space} of $w$, see~\cite[Chapter~4]{Armstrong06}.
This space is endowed with an orientation, by declaring $(\rho(t_1) , \dots , \rho(t_j))$ as a positive ordered basis.  Moreover, $\gamma(t_1,\dots,t_j)$ is full-dimensional in $M(t_1, \allowbreak \dots, \allowbreak t_j)$, by Lemma~\ref{lin_indep}.

\begin{defi}[The sign $\omega$]
 If $u_1,\dots,u_j\in T$ are such that $u_1 \cdots u_j=w$, we define a sign $\omega(u_1,\cdots,u_j)\in \{\pm 1\}$ by the condition that the value is $1$ (resp., $-1$) if $(\rho(u_1),\dots,\rho(u_j))$ is a positive (resp., negative)  basis of $M(t_1,\dots,t_j)$.  If $f = \{ u_1 \succ \dots \succ u_j\}\in\Delta^+$ and $\nc(f)=w$, we denote $\omega(f) = \omega(u_1,\dots,u_j)$.
\end{defi}

Note that in the situation above, $(\rho(u_1),\dots,\rho(u_j))$ is a basis of $M(w)$ by Lemma~\ref{lin_indep}.  So $\omega(u_1,\dots,u_j)$ is well-defined.

\begin{theo} \label{theo:prod}
  Let
  \[
   X(t_1,\dots,t_j) := \big\{ f \in \Delta^+ \; \big| \;  \nc(f)=w \text{ and } \gamma(f) \subset \gamma(t_1,\dots,t_j) \big\}.
  \]
  Then, the cones $\gamma(f)$ for $f\in X(t_1,\dots,t_j) $ are the maximal cones of a simplicial fan with support $\gamma(t_1,\dots,t_j)$, and there holds   
  \begin{align}  \label{exp_prod}
   \dt_1 \cdots \dt_j = \sum_{f\in X(t_1,\dots,t_j)} \omega(f) \cdot \df.
  \end{align}
\end{theo}

We have the following immediate corollary, which describes the structure constants of $\Dual$ with respect to the basis given in Theorem~\ref{theo:basis}.
\begin{coro}
\label{coro:dual_product}
 Let $f_1=\{ t_1 \succ \dots \succ t_i\},f_2=\{ u_1 \succ \dots \succ u_j\}$ be two faces of $\Delta^+$. Then $\df_1 \df_2=0$ unless $(t_1,t_2,\dots,t_i,u_1,u_2,\dots,u_j)$ is a reduced sequence whose product is in $NC$. In that case
 $$\df_1 \df_2 = \sum_{f\in X(t_1,\dots,t_i,u_1,\dots,u_j)} \omega(f) \cdot \df$$.
 \end{coro}

\begin{exam}
Recall that the simple reflections $s_1,\dots,s_n$ are indexed so that our Coxeter element is $c=s_1 \cdots s_n$.  Then
$\gamma(s_1,\dots,s_n)$ contains all positive roots, so it contains all the cones $\gamma(f)$ where $f$ is a $c$-cluster.  By the previous theorem, there holds in $\mathcal{P}$:
\[
  \ds_1\cdots \ds_n=\sum_{f \in \Delta^+_{n-1}} \omega(f) \cdot \df.
\]
\end{exam}

\begin{proof}[Proof of Theorem~\ref{theo:prod}]

Let $(t_1,\dots,t_j)$ be as in the statement of the theorem. It is an ordered face of $\Delta_+$ if and only if we have $t_1\succ \dots \succ t_j$. In this case the statement is clear.

We may now assume that there exists $i$ such that $t_i\prec t_{i+1}$. Then $\Gamma(t_it_{i+1})$ is a rank $2$ reflection subgroup of $W$ with Coxeter generators $t_i$ and $t_{i+1}$, faithfully acting in the plane $P$ spanned by the roots $\rho(t_i)$ and $\rho(t_{i+1})$. By Lemma~\ref{lem:rank2}, its reflection set is $T(t_it_{i+1})=\{v_1=t_i,v_2,\dots,v_m=t_{i+1}\}$, where the corresponding roots  $\rho(v_i)$ are  linearly ordered as in Figure~\ref{fig:fan_dim_2} or its mirror image. It follows that:\begin{itemize}
    \item the cones $\gamma(v_i,v_{i-1})$ for $1 < i \leq m$ form a simplicial fan in dimension $2$, with support $\gamma(v_m,v_1)$, and
    \item the ordered bases $(v_i,v_{i-1})$ for $1 < i \leq m$ and $(v_m,v_1)$ have the same orientation.
\end{itemize} 

We now lift this to dimension $j$. From the relation \[\dt_i\dt_{i+1}=\dv_1\dv_m=-\sum_{i=1}^{m-1} \dv_i\otimes \dv_{i-1} \in R_3\] we get:
\begin{equation}
\label{eq:product_induction}
 \dt_1\cdots \dt_j=-\sum_{i=1}^{m-1} \dt_1\dots \dt_{i-1}(\dv_i\otimes \dv_{i-1}) \dt_{i+2}\cdots\dt_j.
\end{equation}

Now on the right hand side of \eqref{eq:product_induction}, note that all monomials are strictly smaller in lexicographic ordering than $\dt_1\cdots \dt_j$, so we can assume the result holds for them by induction. The partition of $X(t_1,\cdots,t_j)$ into cones is obtained by lifting the $2$-dimensional picture detailed above. The product formula then follows from the fact that each term $f'$ in the sum \eqref{eq:product_induction} has  $\omega(f')=-\omega(f)$ once again by the analysis of the $2$-dimensional case.
\end{proof}

%%%%%%%%%%%%%%%%%%%%%%%%
\section{The Nichols algebra}
\label{sec:shuffle}
%%%%%%%%%%%%%%%%%%%%%%%%

The goal of this section is to build on Remark~\ref{rem:unshuffle}.  Going back to the more natural framework of algebras, rather than coalgebras, this observation suggests the introduction of the {\it twisted shuffle product} that we define below.  It will lead to an alternative construction of the algebra $\mathcal{P}$, as a quotient of a {\it Nichols algebra}. This construction is an alternative way to get properties proved in Section~\ref{sec:propertiesdual} --though not the most natural way when starting from the braid monoid.  It will also connect our work with the Orlik-Solomon algebra.

This realization of the algebra $\Dual(W)$ as a quotient of $\mathcal{N}(W)$ was also obtained independently in Zhang's thesis~\cite{Zhang20}, where he uses this algebra to compute the homology of Milnor fibers.  As explained in Remarks~\ref{rem:vanishingprop} and~\ref{rem:zhang}, his results can be used to get an alternative path to koszulity of $\Alg$ and $\Dual$.

\subsection{Definitions}

%\setmathfont[range={\shuffle}]{XITS Math}

The unsigned shuffle product is common in algebraic combinatorics.  The signed version used here (the product $\shuffle$ defined below) is natural in a geometric context (for example, it defines the wedge product on antisymmetric multilinear maps, or differential forms).

\begin{defi} \label{def:tsh}
 The {\it shuffle product} $\shuffle$ on $\mathcal{T}(\bk^T)$ is defined recursively on the canonical basis by 
  \begin{align*}
   (t_1\otimes \cdots \otimes t_i)  \shuffle (u_1\otimes \cdots \otimes u_j)
   &= 
   t_1 \otimes \Big( (t_2\otimes \cdots \otimes t_i) \shuffle (u_1\otimes \cdots \otimes u_j) \Big) + \\
   (-1)^i & u_1 \otimes \Big( (t_1 \otimes \cdots \otimes t_i) \shuffle (u_2 \otimes \cdots \otimes u_j) \Big).
 \end{align*}Similarly, its twisted analog $\tsh$ is defined by:
 \begin{align*}
   (t_1\otimes \cdots \otimes t_i)  \tsh (u_1\otimes \cdots \otimes u_j)
   &= 
   t_1 \otimes \Big( (t_2\otimes \cdots \otimes t_i) \tsh (u_1\otimes \cdots \otimes u_j) \Big) + \\
   (-1)^i & u_1 \otimes \Big( (t_1^{u_1} \otimes \cdots \otimes t_i^{u_1}) \tsh (u_2 \otimes \cdots \otimes u_j) \Big).
 \end{align*}
These products have the same unit as the usual product of $\mathcal{T}(\bk^T)$.
\end{defi}

For example,
\begin{align*}
  (t_1\otimes t_2) \tsh (u_1\otimes u_2) = &t_1 \otimes  t_2 \otimes u_1 \otimes u_2
   - t_1 \otimes u_1 \otimes t_2^{u_1} \otimes u_2 + \\
   & t_1 \otimes u_1 \otimes u_2 \otimes t_2^{u_1u_2}
   + u_1 \otimes t_1^{u_1} \otimes t_2^{u_1} \otimes u_2 - \\
   & u_1 \otimes t_1^{u_1} \otimes u_2 \otimes t_2^{u_1u_2} + u_1 \otimes u_2 \otimes t_1^{u_1u_2} \otimes t_2^{u_1u_2}.
\end{align*}

\begin{defi}
We denote $\mathcal{N} \subset \mathcal{T}(\bk^T)$ the $\tsh$-subalgebra generated by degree 1 elements (i.e., reflections).  We refer to $\mathcal{N}$ as the {\it Nichols algebra}.
\end{defi}

In general, there is a Nichols algebra associated to each {\it braided vector space} (a general reference on this subject is~\cite{TAK05}).  Nichols algebras are {\it braided Hopf algebras} in the sense of~\cite{MAJ94}, as the braiding intervenes in the compatibility relation between the product and the coproduct.  Only the product is relevant here.  The relevant braiding on $\bk^T$ is a linear endomorphism $\varsigma$ of $\bk^T \otimes \bk^T$ defined by:
\[
  \varsigma(t\otimes u) := - u \otimes t^u.
\]
It is straightforward to check that $\varsigma$ satisfies the Yang-Baxter equation, so that $(\bk^T,\varsigma)$ is a braided vector space.  Some properties of $\mathcal{N}$ follow from~\cite{MS00}, see below. 

Combinatorially, it should be noted that $\varsigma$ is the signed version of the {\it left Hurwitz move}, which consists in replacing a factor $t_i\otimes t_{i+1}$ with $t_{i+1} \otimes t_i^{t_{i+1}}$ in a tensor $t_1\otimes \cdots \otimes t_j$.  A key point of these moves is that the product map $t_1\otimes \cdots \otimes t_j\mapsto t_1\cdots t_j $ is invariant.  Similarly, we get a decomposition of $\mathcal{N}$ that makes it a graded algebra, where the grading takes values in $W\times\mathbb{N}$ (with the obvious monoid structure).

\begin{lemm}
  For $(w,j)\in W\times\mathbb{N}$, let
  \[
    \mathcal{N}_{(w,j)} := \operatorname{Span}_{\bk} \Big\{
    t_1 \tsh \cdots \tsh t_j \; \Big| \; t_1,\dots,t_j\in T, \; t_1 \cdots t_j = w \Big\}.
  \]
  Then we have:
  \[
    \mathcal{N} = \bigoplus_{(w,j) \in W\times\mathbb{N} } \mathcal{N}_{(w,j)}
  \]
  and
  \[ 
    \mathcal{N}_{(w,j)} \tsh \mathcal{N}_{(w',j')} \subset \mathcal{N}_{(w w' , j + j')}.
  \]    
  Moreover, if $\dim \mathcal{N}_{(w,j)}>0 $, then $\lt(w) = j-2i$ for some $i\in\mathbb{N}$.
\end{lemm}

\begin{proof}
  In Definition~\ref{def:tsh}, it is easily seen that all terms in $(t_1\otimes \cdots \otimes t_i)  \tsh (u_1\otimes \cdots \otimes u_j)$ are obtained from $t_1\otimes \cdots \otimes t_i \otimes u_1\otimes \cdots \otimes u_j$ by applying $\varsigma$ on some adjacent pairs.  It follows that all terms are in the same subspace $\mathcal{N}_{(t_1\cdots t_i u_1 \cdots u_j,i+j)}$.  The existence of the grading easily follows.

  If $\dim( \mathcal{N}_{(w,j)} )>0 $, from a nonzero element we get $t_1,\dots, t_j\in T$ such that $t_1\cdots t_j =w$.  Therefore, $\ell_T(w)\leq j$ by definition of $\ell_T$.  That they have the same parity follows from the fact that each reflection has odd Coxeter length.
\end{proof}

%In fact, this grading also holds in $\mathcal{T}(\bk^T)$, either for the usual product or for $\tsh$.  
Below, we can assume that the indices $(w,j)$ always satisfy the condition ensuring that $\dim( \mathcal{N}_{(w,j)} )>0 $, in particular $\lt(w)\leq j$.

Let us mention some interesting properties of $\mathcal{N}$ taken from Milinski and Schneider~\cite{MS00} (these won't be used in the rest of the section).

\begin{prop}[{\cite[Theorem~5.8]{MS00}}]
  For each $w\in W$, let $\mathcal{N}_w = \oplus_{j\in \mathbb{N}} \mathcal{N}_{w,j}$.  The vector spaces $\mathcal{N}_w$ have all the same dimension.
\end{prop}

This result implies that $\mathcal{N}$ is finite-dimensional if and only if $\mathcal{N}_e$ is (where $e$ is the unit of $W$).  Note that $\mathcal{N}_e$ is a subalgebra of $\mathcal{N}$.  It would be very interesting to know if these algebras are indeed finite-dimensional.

\begin{prop}[{\cite[Corollary~5.9]{MS00}}]
  For each $w\in W$, let $x_w = s_1 \tsh s_2 \tsh \cdots $ where $s_1s_2 \cdots $ is a reduced expression for $w$ (as a product of elements in $S$).  Then:
  \begin{itemize}
      \item $x_w$ does not depend on the chosen reduced expression, up to a sign;
      \item $(x_w)_{w\in W}$ is a linear basis of the subalgebra of $\mathcal{N}$ generated by $S$;
      \item we have $x_w x_{w'} = \pm x_{ww'}$ if $\ell_S(ww') = \ell_S(w) + \ell_S(w')$ (where $\ell_S$ is Coxeter length), and $0$ otherwise.
  \end{itemize}
\end{prop}

Note that the subalgebera in the previous proposition somewhat resembles the Nilcoxeter algebra of $W$.

\subsection{The dual algebra as a quotient of the Nichols algebra}

Our goal is to build the algebra $\mathcal{P}(W,c)$ as a quotient of $\mathcal{N}$.  We develop this point of view from scratch, i.e., independently from the results about $\mathcal{P}(W,c)$ obtained above.  We thus obtain a definition and some properties of an algebra $\mathcal{P}'(W,c)$, and only at the end of this section will be discussed the fact that it is isomorphic to $\mathcal{P}(W,c)$.

\begin{lemm}
For each Coxeter element $c$, the subspace $J_c \subset \mathcal{N}$ defined by 
\[
  J_c := \bigoplus_{(w,j)\in W\times \mathbb{N}, \; w \notin NC_j(W,c)} \mathcal{N}_{(w,j)}.
\]
is an ideal.  Moreover, it is generated by the degree 2 elements $t\tsh u$ such that $t,u\in T$ and $tu \notin NC_2(W,c)$.
\end{lemm}

\begin{proof}
  As $J_c$ is the sum of a subset of the homogeneous components, being an ideal amounts to a stability property of the bigrading. 
 
  Let $(w,j), (w',j') \in W \times \mathbb{N}$, such that $\ell_T(w)\leq j$ and $\ell_T(w')\leq j'$.  If $ww'\in NC_{j+j'}$, we have
  \[
    j+j'=\ell_T(ww') \leq \ell_T(w) + \ell_T(w') \leq j+j',
  \] 
  and it follows that $\ell_T(w) = j$ and $\ell_T(w') = j'$.  It also follows that $w$ and $w'$ are below $ww'$ in the absolute order, so $w\in NC_j$ and $w'\in NC_{j'}$.  By contraposition, $w\notin NC_j$ or $w'\notin NC_{j'}$ implies $ww' \notin NC_{j+j'}$.  This property of the bigrading shows that $J_c$ is an ideal.
 
  The fact that $J_c$ is generated by its degree 2 elements follows from the contraposition of the combinatorial property: if $t_1,\dots, t_k \in T$ are pairwise distinct and such that $t_it_j \in NC_2(W,c)$ for all $i<j$, then $t_1 \cdots t_k \in NC_k(W,c)$ and this element is the meet of $t_1,\dots, t_k$ in $NC(W,c)$.  Assume by induction that this holds for $k-1$ (the case $k=2$ is clear).  We thus have $t_2\cdots t_k \in NC_{k-1}(W,c)$ by induction hypothesis.  If $i\geq 2$, from $t_1t_i \in NC_2(W,c)$ and $t_1\neq t_i$, we get $t_i \leq t_1 c$.  As $t_2\cdots t_k$ is the meet of $t_2, \dots ,t_k$ in $NC(W,c)$, we also have $t_2 \cdots t_k \leq t_1 c$.  Eventually, 
  \[
   t_1 \vee \dots \vee t_k 
   =
   t_1 \vee ( t_2 \vee \dots \vee t_k )
   =
   t_1 \vee ( t_2 \cdots t_k )
  \]
  and this is easily seen to be $t_1 \cdots t_k$.
\end{proof}

\begin{defi}
  We define the algebra $\mathcal{P}'(W,c)$ as the quotient $\mathcal{N} / J_c$.  We also define $\mathcal{P}'_w(W,c) \subset \mathcal{P}'(W,c)$ for $w\in NC(W,c)$ as the quotient of $\mathcal{N}_{w,\lt(w)}$ by its intersection with $J_c$ (as a vector space).
\end{defi}

It was mentioned above that $J_c$ is an homogeneous ideal of $\mathcal{N}$.  It follows that the quotient inherits from the grading $\mathcal{N}$, and we immediately get:
\[
  \mathcal{P}'(W,c)
  =
  \bigoplus_{w\in NC(W,c)}
  \mathcal{P}'_w(W,c)
\]
Moreover this is a grading in the sense of Proposition~\ref{prop:refinedgrading}.

Note that in the previous definition, $c$ is not assumed to be a standard Coxeter element, it can be any Coxeter element.  By considering a map $\gamma_w: T\to T$ defined by $\gamma_w(x) = w x w^{-1}$ (and its extension to $\bk^T$), we easily see that $( \gamma_w \otimes \gamma_w ) \circ \varsigma = \varsigma \circ (\gamma_w\otimes \gamma_w)$, and consequently $\gamma_w \circ \tsh = \tsh \circ (\gamma_w\otimes \gamma_w)$, so that $\gamma_w$ can be extended to an automorphism of $\mathcal{N}$ that sends $J_c$ to $J_{\gamma_w(c)}$.  Therefore, we can assume that $c$ is a standard Coxeter element, without loss of generality.  Accordingly, we can use the $c$-compatible reflection ordering on $T$, as in the definition of $\Delta^+(W,c)$.

\begin{lemm}
 If $f=\{t_1\succ \dots \succ t_j\} \in \Delta^+_{j-1}$, we denote $f_{\otimes}  := t_1 \otimes \cdots \otimes t_j$ and $f_{\tsh} = t_1 \tsh \cdots \tsh t_j$.
 Then $f_{\tsh} - f_{\otimes}$ is a linear combination of tensors $f'_{\otimes}$ where $f' \in \Delta^+_{j-1}$ is strictly smaller than $f$ in the lexicographic order.
\end{lemm}

\begin{proof}
 Write $f_{\tsh} = t_1 \tsh (t_2 \tsh \cdots \tsh t_j)$.  Using the definition of $\tsh$, we easily see by induction on $j$ that the tensors appearing in $f_{\tsh}-f_{\otimes}$ have the form $(t_1\otimes t_2 \otimes \cdots \otimes t_i) \otimes t_\ell$ with $0\leq i <n-1$ and $i+1<\ell$.  Since $t_{i+1}\succ t_\ell $, these are lexicographically smaller than $f_{\otimes}$.
\end{proof}

For $f\in\Delta^+$, let $f_{\otimes}^* \in \mathcal{N}^*$ denote the map defined as taking the coefficient of $f_{\otimes}$ (in the expansion with respect to the canonical basis of $\mathcal{T}(\bk^T)$).  This map vanishes on $J_c$, so it is also well-defined on the quotient $\mathcal{P}'$.  We keep the same notation for this quotient map.

\begin{lemm} \label{lemm:lem1}
We have:
\begin{itemize}
    \item the elements $(f_{\tsh})_{f\in\Delta^+}$ are linearly independent in $\mathcal{P}'(W,c)$,
    \item the elements $(f_{\otimes}^*)_{f\in\Delta^+}$ are linearly independent in $\mathcal{P}'(W,c)^*$.
\end{itemize}
In particular, $\dim \mathcal{P}'(W,c) \geq \# (\Delta^+(W,c))$.
\end{lemm}

\begin{proof}
It follows from the previous lemma that the matrix with entries $f^*_{\otimes}(f'_{\tsh})$ for $f,f' \in \Delta^+(W,c)$ is unitriangular, hence invertible.
\end{proof}

\begin{prop} \label{prop:nicholsPbasis}
  The family $(f_{\tsh})_{f \in \Delta^+(W,c)}$ is a $\bk$-linear basis of $\mathcal{P}'(W,c)$.  In particular, $\dim \mathcal{P}'(W,c) = \# \Delta^+(W,c) $.  
\end{prop}

\begin{proof}
  Using the homogeneous decomposition, it suffices to show that the elements $f\in \Delta^+(W,c)$ with $\nc(f) = w$ form a basis of $\mathcal{P}'_w(W,c)$.  By the previous lemma, it suffices to show that these elements are a generating family.  
  
  Let us first show that Relations~\eqref{it:rel3} in Theorem~\ref{theo:dual_presentation} also hold in $\mathcal{P}'(W,c)$.  If $u_1,\dots,u_m$ are defined as in Lemma~\ref{lem:rank2}, we have $u_i\tsh u_{i-1} = u_i\otimes u_{i-1} - u_{i-1}\otimes u_{i-2}$ (taking indices modulo $m$), it follows: $\sum_{i=1}^m u_i \tsh u_{i-1} = 0$.  Now, we can use the argument in Proposition~\ref{theo:basis}: these relations permit to rewrite each monomial as a combination of decreasing monomials.
\end{proof}

\begin{rema}
  It would be very interesting to find a presentation of $\mathcal{N}$, or at least to know if it is quadratic.  If it is the case, $\mathcal{P}'(W,c)$ is also a quadratic algebra (as we have shown that $J_c$ is a quadratic ideal).  As such, the basis $(f_{\tsh})_{f\in \Delta^+(W,c)}$ guarantees that $\mathcal{P}'(W,c)$ is a Koszul algebra, as it is a Poincaré-Birkhoff-Witt basis (see~\cite[Section~5]{Priddy}).  We would also get the presentation $\mathcal{P}'(W,c)$ as in Theorem~\ref{theo:dual_presentation}.  
\end{rema}

\begin{theo} \label{theo_iso_nicholsquotient}
  There is a well-defined isomorphism $\Psi$ from $\mathcal{P}(W,c)$ to $\mathcal{P}'(W,c)$ such that for all $t_1,\dots,t_j\in T$, we have:
  \begin{equation} \label{eq_iso_shuffle1}
    \Psi( \dt_1 \cdots \dt_j ) = t_1 \tsh \cdots \tsh t_j.
  \end{equation}
\end{theo}

\begin{proof}
  The generating sets $T$ of $\mathcal{P}'(W,c)$ and $\dT$ of $\Dual$ are in bijection.  We have already seen in the proof of Proposition~\ref{prop:nicholsPbasis} that Relations~\eqref{it:rel3} in Theorem~\ref{theo:dual_presentation} hold in $\mathcal{P}'(W,c)$.  Relations~\eqref{it:rel1} and \ref{it:rel2} hold as well, since:
  \begin{itemize}
    \item if $t\in T$, $t\tsh t = t\otimes t - t\otimes t = 0$;
    \item if $t,u\in T$ are such that $tu\not \leq_T c$, we have $t\tsh u \in \mathcal{N}_{(tu,2)} \subset J_c$ so $t\tsh u =0$ in $\mathcal{N}/J_c$.
  \end{itemize}
  It follows that the map $\Psi$ is well-defined and surjective.  By the previous lemma, the two algebras have the same dimension, and consequently the map is an isomorphism. 
\end{proof}

Note that the previous proof uses the fact that $\dim \mathcal{P}(W,c) = \# \Delta^+(W,c)$, so it relies on the results from Section~\ref{sec:koszulity} via the vanishing property (Proposition~\ref{prop:vanishingprop}.

\subsection{A parallel with the Orlik-Solomon algebra}  
Some properties of $\Dual$ obtained in Section~\ref{sec:propertiesdual} presents some similarity with the algebra $\mathcal{OS}=\mathcal{OS}(W)$ from~\cite{OS80}, as noted by Zhang~\cite{Zhang20}.  Let us recall its definition.

For each $t\in T$, let $H_t:=\ker(t-I) \subset \mathbb{R}^n$, and
\[
  M(W) := \mathbb{C}^n \backslash \big( \bigcup_{t\in T} H_t\otimes\mathbb{C} \big).
\]
This $M(W)$ is a classifying space for the pure braid group  $\bB(W) / W$, by a result of Deligne conjectured by Brieskorn (for example, see~\cite{Bes15} and references therein).

For each $t\in T$, let $\alpha_t \in (\mathbb{C}^n)^*$ such that $\ker (\alpha_t) = H_t \otimes \mathbb{C}$, and $\omega_t = \frac{\dif \alpha_t}{2i\pi \alpha_t}$.  The algebra $\mathcal{OS}(W)$ can be defined over $\mathbb{Z}$ as the ring of differential forms generated by the 1-forms $(\omega_t)_{t\in T}$, together with the constant $0$-form 1 as a unit.  It is isomorphic to the cohomology ring $H^*(M(W),\mathbb{Z})$ via the map sending a closed differential form to its de Rham cohomology class.  This algebra can also be seen as the cohomology ring of the pure braid group, since $M(W)$ is a classifying space.

\begin{defi} 
Let 
\[
L(W) := \Big\{  \bigcap_{t\in T'} H_t \; \Big| \;  T'\subset T \Big\}.
\]
Endowed with reverse inclusion, it is a geometric lattice called the {\it intersection lattice} of $W$.
\end{defi}

For $x\in L$, let $\mathcal{OS}_x$ be the subspace of $\mathcal{OS}$ linearly generated by $\omega_{t_1}\wedge \cdots \wedge \omega_{t_i}$ where $t_1,\dots,t_i\in T$ are such that $H_{t_1} \cap \cdots \cap H_{t_i} = x$.  Then $\mathcal{OS} = \bigoplus_{x\in L} \mathcal{OS}_x$, and this decomposition is compatible with the product in the sense that $\mathcal{OS}_x \wedge \mathcal{OS}_{x'} \subset \mathcal{OS}_{x\cap x'}$.  Moreover, $\dim( \mathcal{OS}_x ) = \mu_L(x)$ ($\mu_L$ is the Möbius function of $L(W)$).  We refer to~\cite[Section~2]{OS80} for all these properties.
The decomposition of $\Dual$ in Proposition~\ref{prop:refinedgrading}, together with Corollary~\ref{coro:dim_mu}, thus present a striking similarity with the results about $\mathcal{OS}$ just mentioned: we just replace the intersection lattice $L(W)$ and its Möbius function with the noncrossing partition lattice $NC(W)$ and its Möbius function. \smallskip

Our goal here is to show that $\mathcal{OS}$ is also a quotient of $\mathcal{N}$, just like $\Dual$.  We use a definition of $\mathcal{OS}$ in terms of shuffles taken from~\cite[Section~3]{OS80}.  In particular, this algebra will be defined over $\bk$ to be consistent with other algebras considered in this work.  
Recall that that we have $\bigwedge (\bk^T) \subset \mathcal{T}(\bk^T)$, by seeing the exterior algebra as the space of antisymmetric tensors.  Moreover, the wedge product on the exterior algebra identifies with $\shuffle$.  Define a map $\lambda : \mathcal{T}(\bk^T) \to \mathcal{T}(\bk^{L\backslash\{0\}}) $ by:
\[
   \lambda (t_1 \otimes \cdots \otimes t_j)
   =
   \begin{cases}
     H_{t_1} \otimes \big(H_{t_1} \cap H_{t_2}\big) \otimes \cdots \otimes \big(\bigcap\limits_{i=1}^j H_{t_i}\big)
     & \text{ if } \bigcap\limits_{i=1}^j H_{t_i} \neq \{0\}, \\
     0 & \text{ otherwise.}
   \end{cases}
\]
Following~\cite{OS80}, $\mathcal{OS}$ can be defined as the quotient $\bigwedge (\bk^T)/I$ where $I$ is the ideal 
\[
   I := \bigwedge (\bk^T) \cap \lambda^{-1}(0).
\]
Alternatively, $\mathcal{OS}$ is identified with $\lambda(\bigwedge (\bk^T))$, with a product defined by $\lambda(u) \cdot \lambda(v) = \lambda( u \shuffle v )$.  The fact that this is well-defined relies on the identity $\lambda(u\shuffle v) = \lambda( \lambda(u) \shuffle \lambda(v) )$, see~\cite{OS80} for details.  To avoid confusion, we keep the notation $\omega_t$ for the generators of $\mathcal{OS}$, and the product is denoted $\wedge$ as for differential forms.

\begin{theo}  \label{theo:nichols_os}
There is a well-defined surjective map $\mathcal{N} \to \mathcal{OS}$ defined by $t\mapsto \omega_t$.  In particular, $\mathcal{OS}$ is a quotient of $\mathcal{N}$.
\end{theo}

\begin{lemm}
 For $u,v\in \mathcal{T}(\bk^T)$, we have $\lambda(u\shuffle v) = \lambda( u \tsh v)$.
\end{lemm}

\begin{proof}
 For $t,u\in T$, it is easily checked that $\{\rho(t),\rho(u)\}$ generates the same linear subspace as $\{\rho(t),\rho(tut)\}$.  By taking the orthogonal subspaces,  we have $H_t \cap H_u = H_t \cap H_{tut}$.
 It follows $\lambda(t\shuffle u) = \lambda( t\tsh u )$, and more generally we can ignore the conjugations when computing $\lambda( (t_1\otimes \cdots\otimes t_i) \tsh (u_1\otimes \cdots\otimes u_j) )$.  We thus have
 \[
   \lambda( (t_1\otimes \cdots\otimes t_i) \tsh (u_1\otimes \cdots\otimes u_j) )
   =
   \lambda( (t_1\otimes \cdots\otimes t_i) \shuffle (u_1\otimes \cdots\otimes u_j) ),
 \]
 and the result follows.
\end{proof}

\begin{proof}[Proof of Theorem~\ref{theo:nichols_os}]
The construction of $\mathcal{OS}$ described above shows that it is generated by the elements $\lambda( t_1 \shuffle \cdots \shuffle t_i)$ for $t_1,\dots,t_i\in T$, with the product of $\lambda( t_1 \shuffle \cdots \shuffle t_i)$ and $\lambda( u_1 \shuffle \cdots \shuffle u_j)$ given by $\lambda( t_1 \shuffle \cdots \shuffle t_i \shuffle u_1 \shuffle \cdots \shuffle u_j)$.  The previous lemma says that $\shuffle$ can be replaced with $\tsh$ everywhere, and the result follows.
\end{proof}

\begin{rema} \label{finalremarksec7}
It would be very interesting to find an explicit description of the kernel of the map $\mathcal{N} \to \mathcal{OS}$, for example by giving an explicit set of generators.  Clearly, this kernel contains the tensors $t_1 \tsh \dots \tsh t_k$ where $\{t_1,\dots,t_k\}$ is a dependent set (i.e., $\dim(H_{t_1} \cap \dots \cap H_{t_k}) > n-k$).
\end{rema}

%%%%%%%%%%%%%%%%%%%%%%%%
\section{Cyclic action on the algebras and homology of the noncrossing partition lattice}
\label{sec:cyclic_action}
%%%%%%%%%%%%%%%%%%%%%%%%

Let $Z\subset W$ be the cyclic group generated by $c$, and let $R$
denote its character ring (i.e., the ring of functions $Z \to
\mathbb{C}$).  The group $Z$ acts on $T$ via $c\cdot t = c t c^{-1}$,
and also on $\bT$ and $\dT$ via the natural identification $T\simeq \bT
\simeq \dT$.  We denote $\langle c^i \rangle \subset Z$ the subgroup
generated by $c^i$.  When $Z$ acts on some object $X$ (a set or a vector
space), we denote $X^{\langle c^i \rangle}$ the subobject of fixed
points and $\operatorname{ch}_Z(X) \in R$ the associated character.  The
evaluation of the character $\operatorname{ch}_Z(X)$ at $c^i\in Z$ is
the trace of $c^i$ acting on $X$, denoted $\tr(c^i,X)$.

The goal of this section is to give formulas for the characters of $Z$ acting on $\Alg(W)$ and $\Dual(W)$, moreover we prove that the top degree component of $\Dual(W)$ is isomorphic to the homology of $NC(W)$ as a $k[Z]$-module.  This gives an alternative proof of a result of Zhang~\cite{Zhang20}.

\subsection{Refinement of the relation between Hilbert series }

The Hilbert series of a graded algebra can be naturally refined to take
into account the action of $Z$.  Explicitly, we define:
\[
   \operatorname{Hilb}_Z( \Alg(W) , q )
   =
   \sum_{n \geq 0} \operatorname{ch}_Z (\Alg_n (W)) q^n,
\]
and similarly for $\Dual(W)$.

\begin{prop}
We have the following identity in $R[[q]]$:
\[
   \operatorname{Hilb}_Z( \Alg(W) , q )
   \operatorname{Hilb}_Z( \Dual(W) , -q )
   = 1
\]
where $1$ is the trivial character of $Z$.
\end{prop}

\begin{proof}
   Using the graduation, the exact complex
in~\eqref{other_free_resolution} can be splitted as a direct sum of
exact complexes.  We find that the complex
   \[
     0 \lra \Alg_0(W) \otimes \Dual^*_n(W) \stackrel{\eth_{n}}{\lra}
\dots
     \stackrel{\eth_{1}}{\lra}
     \Alg_n(W) \otimes \Dual^*_0(W)
     \lra 0
   \]
   (where the general term is $\Alg_i(W) \otimes \Dual^*_j(W)$ with
$i+j=n$) is exact for $n>0$.  Moreover it is clear that the maps
$\eth_i$ commute with the action of $Z$.  By the Hopf trace
formula~\cite{Wac07}, the alternating sum of the summands in this complex
is $0$ in $R$.  This precisely says that the coefficient of $q^n$ in $ 
\operatorname{Hilb}_Z( \Alg(W) , q ) \operatorname{Hilb}_Z( \Dual(W) ,
-q ) $ is $0$.  As $\operatorname{ch}_Z(\Alg_0(W)) =
\operatorname{ch}_Z(\Dual_0(W)) = 1$, this completes the proof.
\end{proof}

Note that the $q$-coefficientwise evaluation of the series
$\operatorname{Hilb}_Z( \Alg(W) , q ) \in R[[q]]$ at $c^i\in Z$ is:
\[
   \sum_{k\geq 0} \tr( c^i , \Alg_k(W) ) q^k.
\]

\subsection{Cyclic action on our algebras}

The action of $Z$ permutes the elements in $\Dbm(W)$.  We first characterize the fixed points of $\langle c^i \rangle$.

\begin{prop}
  $\Dbm(W)^{\langle c^i \rangle}$ is a Garside monoid having $\bc$ as
Garside element, and its underlying lattice is $\bNC(W)^{\langle c^i
\rangle} \simeq NC(W)^{\langle c^i \rangle}$.
\end{prop}

\begin{proof}
An element of $Z$ acts by automorphism on $\Dbm(W)$, moreover this action preserves the Garside element $\bc$.  It is a straightforward consequence of the axioms of Garside monoids in Definition~\ref{defgarside} that the fixed-point set of such an automorphism is a Garside submonoid with the same Garside element.
\end{proof}

We can thus use the same argument as in the case of $\Dbm(W)$ and get
the following:

\begin{coro}
  The growth function of $\Dbm(W)^{\langle c^i \rangle}$ is:
  \[
    \sum_{ \bb \in \Dbm^{\langle c^i \rangle}}
    q^{|\bb|}
    =
     \Big( \sum_{ w \in NC(W)^{\langle c^i \rangle} }
    \mu(w) q^{\lt(w)}
     \Big)^{-1} .
  \]
  where $\mu$ is the Möbius function of $NC(W)^{\langle c^i \rangle}$ and $|\bb|$ is the length in $\Dbm(W)$.
\end{coro}

Note that in the previous statement, $|\bb|$ is the length function of
elements $\bb \in \Dbm(W)$ (which is in general an integer multiple of
the length function in the submonoid $\Dbm(W)^{\langle c^i \rangle}$).

\begin{rema}
   There are quite a few results about the action of $Z$ on $NC(W)$ in the literature.  For example, see~\cite[Section~5.4.2]{Armstrong06}.  However, to our knowledge there is no explicit description of the subposets $NC(W)^{\langle c^i \rangle}$.  It turns out that in each case there is an isomorphism $NC(W)^{\langle c^i \rangle} \simeq NC(W')$ for a smaller reflection group $W'$ though we don't have any conceptual explanation.  
\end{rema}

\begin{coro}[{\cite[Theorem 4.60]{Zhang20}}]
  We have:
  \[
    %\operatorname{Hilb}_Z( \Dual(W) , q ) \big( c^i \big )
    \sum_{k=0}^n  \tr( c^i , \Dual_k(W) ) q^k
    =
    \sum_{ w \in NC(W)^{\langle c^i \rangle} }
    \mu(w) (-q)^{\lt(w)}
  \]
  where $\mu$ is the Möbius function of $NC(W)^{\langle c^i \rangle}$.
\end{coro}

\begin{proof}
   As the action of $Z$ permutes the elements of $\Dbm(W)$, the character
of the action is obtained by counting fixed points.  More precisely, the
growth function in the previous corollary is the coefficientwise
evaluation at $c^i$ of $\operatorname{Hilb}_Z( \Alg(W) , q)$.  The
result thus follows from the previous corollary and the relation between
$\operatorname{Hilb}_Z( \Alg(W) , q)$ and $\operatorname{Hilb}_Z(
\Dual(W) , q)$.
\end{proof}

\subsection{Homology of the noncrossing partition lattice}

Let us begin with a few preliminaries related with the shuffle product. 
Note that the shuffle product $\tsh$ has a natural extension to
$\mathcal{T}(\bk^W)$ using the same formula as in~\eqref{def:tsh}.

\begin{defi} \label{def_nabla}
  We define a map $\nabla : \mathcal{T}(\bk^W)\to \mathcal{T}(\bk^W)$ by:
  \[
    \nabla(t_1 \otimes \dots \otimes t_k)
    =
    \sum_{i=1}^{k-1} (-1)^i t_1 \otimes \dots \otimes t_{i-1} \otimes
(t_i t_{i+1} ) \otimes t_{i+2} \otimes \dots \otimes t_k.
  \]
  Note that this is $0$ if $k=1$.  The $i$th term in this sum will be
called the $i$th {\it contraction} of $t_i \otimes t_{i+1} $.
\end{defi}

\begin{lemm}  \label{nablavanish}
   For $x,x' \in \mathcal{T}(\bk^T)$ such that $x$ is homogeneous, we
have:
   \[
     \nabla (x \tsh x')
     =
     \nabla (x) \tsh x' + (-1)^{\deg(x)} x \tsh \nabla(x').
   \]
   In particular, $\nabla$ vanishes on $\mathcal{N}(W)$.
\end{lemm}

\begin{proof}
   Assume $x = t_1 \otimes \dots \otimes t_i$     and $x' = u_1 \otimes
\dots \otimes u_j$.  We expand $\nabla (x \tsh x')$ using the definition
of the two operators.  Each term thus corresponds to the choice a
shuffle and the choice of a contraction in this shuffle.

   Consider a shuffle containing a factor $t_k^{u_1 \cdots u_{\ell-1}}
\otimes u_\ell$, and the other shuffle with the factor $t_k^{u_1 \cdots
u_{\ell-1}} \otimes u_\ell$ replaced with $ u_\ell \otimes t_k^{u_1
\cdots u_\ell}$ (which has therefore opposite sign).  We easily see that
the two contractions of these pairs give the same term with opposite
signs.

   Now consider the contractions only involving $t_k\otimes t_{k+1}$,
$1\leq k \leq i-1$.  After checking the signs, it is straightforward to
obtain $\nabla (x) \tsh x'$.  Similarly, the remaining terms give
$(-1)^{\deg(x)} x \tsh \nabla(x')$.
\end{proof}

\begin{rema}
   It is straightforward to prove $\nabla^2=0$.  This means that
$(\mathcal{T}(\bk^W), \tsh , \nabla)$ is a dg-algebra.  However, this is
not particularly relevant here.
\end{rema}

\begin{defi}
  We define a map $\varDelta : \mathcal{T}(\bk^W) \to \mathcal{T}(\bk^W)$
by:
  \[
    \varDelta( w_1 \otimes \dots \otimes w_k ) = \sum_{i=1}^k (-1)^i w_1
\otimes \dots \otimes w_{i-1} \otimes w_{i+1} \otimes \dots \otimes w_k.
  \]
  The $i$th term will be called the $i$th {\it deletion}.  Eventually, we
define $\xi : \mathcal{T}(\bk^W) \to \mathcal{T}(\bk^W)$ by:
  \[
    \xi( w_1 \otimes \dots \otimes w_k )
    =
    w_1 \otimes (w_1w_2) \otimes \dots \otimes (w_1w_2 \cdots w_{k-1}).
  \]
\end{defi}

Following the convention in Section~\ref{sec:shuffle}, we denote
$\mathcal{T}_{w,k}(\bk^W)$ the subspace of $\mathcal{T}(\bk^W)$
generated by tensors $w_1 \otimes \dots \otimes w_k$ such that $w=w_1
\cdots w_k$.  Note that $\xi$ restricted to $\mathcal{T}_{w,k}(\bk^W)$
is an isomorphism on its image, and its inverse is given by:
  \[
    \xi^{-1} ( w_1 \otimes \dots \otimes w_{k-1} )
    =
    w_1 \otimes (w_1^{-1} w_2) \otimes \dots \otimes (w_{k-1}^{-1} w).
  \]
Note also that $\xi$ commutes with the action of $Z$.

\begin{lemm} \label{xinabla}
   We have $\xi \circ \nabla = \varDelta \circ \xi$ on
$\mathcal{T}(\bk^W)$.  In particular, $\varDelta$ vanishes on
$\xi(\mathcal{N}(W))$.
\end{lemm}

\begin{proof}
The first statement is straightforward and details are omitted.  The
second statement follows because $\nabla$ vanishes on $\mathcal{N}(W)$, by Lemma~\ref{nablavanish}.
\end{proof}

We give a brief definition of homology of posets, and refer to
\cite{Wac07} for details.

\begin{defi}
   For $-1 \leq m \leq n-2$, let $\mathcal{C}_m$ denote the vector space
freely generated by strict chains $w_0 < \dots < w_m$ in $NC(W) -
\{e,c\}$ (i.e., $NC(W)$ with its top and bottom elements removed).  Let
$\varDelta_m : \mathcal{C}_m \to \mathcal{C}_{m-1}$ be the operator
defined by
   \[
     \varDelta_m ( w_0 < \dots < w_m ) = \sum_{i=0}^m (-1)^i  ( w_0 <
\dots < w_{i-1} < w_{i+1} < \dots < w_m).
   \]
   It is straightforward to check that $\varDelta_{m} \circ
\varDelta_{m+1} = 0$.  The $m$th {\it reduced homology group} $\tilde
H_m(NC(W))$ of $NC(W)$ is defined as $\ker(\varDelta_m) /
\im(\varDelta_{m+1} )$.
\end{defi}

Note that $\varDelta_m$ is essentially the restriction of $\varDelta$,
upon identifying the chain $w_0 < \dots < w_m$ with the tensor $w_0
\otimes \dots \otimes w_m$.

As the action of $Z$ on $NC(W)$ preserves the order relation,
$\mathcal{C}_m$ is a $k[Z]$-module.  The map $\varDelta_m$ clearly
commutes with the action of $Z$, so that $\tilde H_m(NC(W))$ also
inherits a structure of $k[Z]$-module.

As $NC(W)$ is a {\it Cohen-Macaulay poset} (see~\cite{Wac07}), it has a
unique nonzero reduced homology group in top degree, namely the space: 
\[
  \tilde H_{n-2} (NC(W))
  = 
  \ker (\varDelta_{n-2})
  \subset 
  \mathcal{C}_{n-2}.  
\]
By {\it Philip Hall's theorem} (see~\cite{Wac07}), $\dim
\tilde H_{n-2} (NC(W))$ is the Möbius invariant of $NC(W)$, therefore it
is equal to $\dim \Dual_{n}(W)$.

To define explicitly the $k[Z]$-isomorphism between $\tilde H_{n-2}
(NC(W))$ and $\Dual_{n}(W)$, we use again the Nichols algebra and
identify $\Dual_{n}(W)$ to $\mathcal{N}_{c,n}(W)$ as $k[Z]$-modules.

\begin{theo}[Zhang]
   The map $ \xi : \mathcal{N}_{c,n}(W) \to \tilde H_{n-2}(NC(W))$
defined by
   \[
     t_1 \otimes \dots \otimes t_n
     \mapsto
     ( t_1 , t_1t_2 , \dots , t_1\cdots t_{n-1} )
   \]
   is an isomorphism of $Z$-module.
\end{theo}

(Note that the formula defines a map on $\mathcal{T}_{c,n}(W)$, the map on $\mathcal{N}_{c,n}(W)$ is obtained by restriction.) 

\begin{proof}
   We have noted above that the restriction of $\xi$ to an homogeneous component is injective, and that $\xi$ commutes with the action of $Z$.  As both spaces have the same dimension (the Möbius invariant of $NC(W)$), it suffices to show that $\xi(\mathcal{N}_{c,n}(W)) \subset \tilde H_{n-2}(NC(W))$.     The inclusion $\xi(\mathcal{N}_{c,n}(W)) \subset \mathcal{C}_{n-2}$ is clear.  From Lemma~\ref{xinabla}, it follows that $\varDelta_{n-2}$ vanishes on $\xi(\mathcal{N}_{c,n}(W))$.  So we have indeed $\xi(\mathcal{N}_{c,n}(W)) \subset \tilde H_{n-2}(NC(W))$.
\end{proof}

%%%%%%%%%%%%%%%%%%%%%%%%
\section{Perspectives}
\label{sec:questions}
%%%%%%%%%%%%%%%%%%%%%%%%

In this final section, we discuss possible extensions of the present work to other kinds of braid groups.

\subsection{Finite well-generated complex reflection groups}

Bessis defined in~\cite{Bes15} the dual braid monoid associated to such a group.  This is again a quadratic monoid defined using the noncrossing partition lattice.  We conjecture that its monoid algebra is a Koszul algebra. Although there is no analog of the cluster complex in this setting, we can define a similar-looking simplicial complex: its facets are given by the decreasing factorizations of a Coxeter element, using the total orders introduced by Mühle~\cite{Muh15} (these total orders were introduced in order to prove the shellability of the noncrossing partition lattice using the adequate analog of Definition~\ref{def:compatiblereforder}).  We conjecture that the complex obtained in this way is shellable and can be used as in Section~\ref{sec:koszulity} to show that the dual braid monoid algebra is again a Koszul algebra.

\subsection{Free groups}

Despite the fact that free groups have an elementary structure, interesting questions appear when we consider them as Artin groups and study them from this perspective.  In this vein, Bessis~\cite{Bes06} gave both a topological and an algebraic definition of the associated dual braid monoids.  There is again a related noncrossing lattice, defined in terms of loops with no self-intersection.  Note that here the monoid is not locally finite (it has infinite cardinality in each degree).  Despite this technical problem, it could be interesting to investigate the existence of a simplicial complex that plays the role of the cluster complex, and leads to a minimal resolution as in Section~\ref{sec:koszulity}. 

\subsection{Affine Weyl groups}

The next step might be to consider the Artin group associated to affine Weyl groups.  Recent progress about these is done in \cite{McS17,PS19}, and Garside theory is an important tool.  The $K(\pi,1)$ problem is solved~\cite{PS19}, and the dual presentation is known (see~\cite[Theorem~C]{McS17} and~\cite[Theorem~8.16]{PS19}).  This means we can consider the dual braid monoid as a submonoid of the braid group, and its defining relations are again quadratic.  On the other side, the clusters in affine type can be related with noncrossing partitions using representation of quivers by work of Ingalls and Thomas~\cite{IngallsThomas}.  Following the lines of the finite type case, it might be possible to use clusters to show koszulity of the affine type dual braid monoid algebra. 

%%%%%%%%%%%%%%%%%%%%%%%%
\section*{Acknowledgements}
%%%%%%%%%%%%%%%%%%%%%%%%

The authors wish to thank the Erwin Schr\"odinger Institute for its 2017 program ``Algorithmic and Enumerative Combinatorics'' where this collaboration started, and Jang Soo Kim for early discussions.  We thank Vic Reiner for his suggestions about the Orlik-Solomon algebras and for drawing our attention to Yang Zhang's thesis~\cite{Zhang20}. Both authors are partially supported by the French ANR grant COMBIN\'E (19-CE48-0011).

\appendix

%%%%%%%%%%%%%%%%%%%%%%%%
\section{The cluster complex.} 
\label{sec:clustercomplex}
%%%%%%%%%%%%%%%%%%%%%%%%

As explained in the introduction, the cluster complex was introduced by Fomin and Zelevinsky~\cite{FZ03}, via a compatibility relation on the set of almost positive roots (faces are the sets of pairwise compatible elements).  A reformulation was given by Brady and Watt~\cite{BWLattice}, using reflection orderings.  It should be noted that their construction also extends the definition to non-crystallographic finite Coxeter groups (those have thus an associated cluster complex, despite the fact that there is no corresponding cluster algebra).  A further extension was done by Reading~\cite{Rea07}: while the complex in~\cite{FZ03,BWLattice} was realized using the {\it bipartite Coxeter element}, he showed that we can start with any standard Coxeter element instead.  Thus, there is a complex $\Delta(W,c)$ defined for $W$ and any standard Coxeter element $c$.  It will be simply denoted $\Delta$.  It can again be reformulated using reflection orderings, this essentially follows from~\cite{CLS14,BJV18}.  

The goal of this appendix is to give definitions that are self-contained and suited for our purpose.  Everything follows from the results in bibliography, but we sketch some proofs for convenience.

We only need to consider a subcomplex $\Delta^+ \subset \Delta$, the {\it positive part} of $\Delta$.  Therefore, it will be convenient to use reflections as vertices, rather than positive roots.  

\begin{defi} \label{def:clust1}
 A tuple of $n$ reflections $t_1,\dots ,t_n \in T$ is called a {\it positive $c$-cluster} if:
 \begin{itemize}
   \item $c = t_1 \cdots t_n$,
   \item and for any $i\neq j$, we have $\langle\rho(t_i)|\rho(t_j)\rangle \geq 0$.
 \end{itemize}
 (Recall that $\rho(t)$ is the positive root associated to $t$.)  Then, $\Delta^+$ is the simplicial complex having positive $c$-clusters as its facets.   We denote by $\Delta^+_i$ the set of $i$-dimensional faces in this complex (where dimension is cardinality minus 1).  By convention, $\Delta^+_{-1} = \{\varnothing\}$.
\end{defi}

This definition essentially comes from~\cite{BWLattice} in the case of the bipartite Coxeter element, and from~\cite{BJV18} for the case of other standard Coxeter elements.  It is equivalent to the definition from~\cite{Rea07}.

\begin{rema}
  The combinatorial structure of $\Delta^+$ depend on the chosen standard Coxeter element.  Two examples of non-isomorphic complexes $\Delta^+$ in type $A_3$ will be given in Appendix~\ref{sec:cluster_examples}.  On the other side, Reading~\cite{Rea07} has shown that this does not happen for the complex $\Delta$ (that we did not define here): all choices of a standard Coxeter element give the same simplicial complex up to isomorphism.
\end{rema}

It is apparent from the previous proposition that the cluster complex is related with noncrossing partitions: each face $\{t_1 , \dots t_k \}$, as a subset of a facet, can be realized as a subword of a minimal reflection factorization of $c$.  The following definition is thus valid:

\begin{defi}
  Each face $f=\{t_1, \dots, t_k\} \in \Delta^+$ can be indexed so that $t_1 \cdots t_k \in NC_k$, and this (well-defined) element is denoted $\nc(f):=t_1\cdots t_k$.  Following the usual convention, we denote $\bnc(f)\in\Dbm$ the corresponding simple braid.
\end{defi}

An important property is the following:

\begin{prop}  \label{prop:ball}
  The topological realization of $\Delta^+$ is an $n-1$-dimensional ball.  %In particular, $\Delta^+$ is a pure simplicial complex.  
\end{prop}

\begin{proof}
 In the bipartite case, this is an immediate consequence of the geometric realization of $\Delta$ via the cluster fan.  Indeed, it follows that $\Delta$ has the topological type of a $n-1$-dimensional sphere.  The subcomplex $\Delta^+$ is obtained by removing the unique facet containing no positive roots.  Topologically, it is thus a ball of the same dimension. 
 
 The case of non bipartite Coxeter elements follows, as the combinatorial type of $\Delta$ does not depend on the chosen Coxeter element.  Alternatively, we can use the non-bipartite cluster fan built in~\cite{RS09}, and the same argument as in the bipartite case.
\end{proof}
 
We also need a similar result about some subcomplexes $\Delta^+(w)$.  In the bipartite case, they were introduced in~\cite{BWLattice}.

\begin{defi} \label{def:parabolicsubcomplex}
For each $w \in NC$ with $w\neq 1$, let $\Delta^+(w)$ denote the full subcomplex of $\Delta^+$ with vertex set $\{t\in T\; :\; t\leq_T w\}$. 
\end{defi}

\begin{prop} \label{prop:deltaplusw}
If $1\leq j \leq n$ and $w \in NC_j$, the topological realization of $\Delta^+(w)$ is a $j-1$-dimensional ball.  In particular, it is contractible.
\end{prop}

\begin{proof}
By Proposition~\ref{parabolic_standard_coxeter_element}, $w$ is a standard Coxeter element of the parabolic subgroup $\Gamma(w)$.  Therefore, there is a complex $\Delta^+$ associated to $(\Gamma(w),w)$.  It can be naturally identified with the complex $\Delta^+(w)$ defined above.  The result thus follows from Proposition~\ref{prop:ball}.  (For more details, see also \cite[Section~8.2]{BJV18}, and references therein).
\end{proof}

Note that the previous property cannot be extended to $w=1$, i.e., $j=0$.  Indeed, that would give the empty simplicial complex, which is not contractible.

To finish this appendix, we explain how the facets of $\Delta^+$ can be characterized via decreasing factorizations of $c$ as a product of $n$ reflections, for some total order $\prec$ associated to the Coxeter element $c$.  We refer to ~\cite{BWLattice} in the case of a bipartite Coxeter element, and ~\cite{CLS14} for other standard Coxeter elements. 

\begin{defi}[Dyer~\cite{Dye93}]
  A total order $\prec$ on $T$ is called a {\it reflection ordering} if the following holds: for any rank 2 parabolic subgroup $P\subset W$ with  $P\cap T = \{u_1,\dots,u_m\}$ indexed as in Lemma~\ref{lem:rank2}, we have either
  \[
    u_1 \prec u_2 \prec \dots \prec u_m
        \quad \text{ or } \quad 
    u_m \prec u_{m-1} \prec \dots \prec u_1.
  \]
\end{defi}

We have seen that there are two valid indexing of each set $P\cap T$, which are reverse of each other.  Note that the above definition does not depend on which indexing is used.

The next definition was introduced in the context of combinatorial topology, to prove the shellability of $NC$.

\begin{defi}[Athanasiadis, Brady and Watt~\cite{ABW07}]  \label{def:compatiblereforder}
A reflection ordering $\prec$ on $T$ is {\it compatible} with a standard Coxeter element $c$ if the following holds: for each $w\in NC_2$, with $\TT(w) = \{ u_1 , \dots , u_m  \}$ indexed so that $u_1 \prec \dots \prec u_m$, we have $w = u_i u_{i-1}$ for $1\leq i \leq m $ (where $u_0=u_m$).
\end{defi}

To explain this definition, a few remarks are necessary.  Let $w$ and $\{u_1,\dots,u_m\}$  and $\TT(w) = \{u_1,\dots,u_m\}$ as in Lemma~\ref{lem:rank2}.  By Proposition~\ref{parabolic_standard_coxeter_element}, we have either $w=u_1 u_m$ or $w=u_m u_1$.  It follows that: either $w = u_i u_{i-1}$ for $1\leq i \leq m $, or $w = u_{i-1} u_i $ for $1\leq i \leq m $.  Note that the two situations are mutually exclusive if $m\geq 3$, but both hold if $m=2$.  The compatibility of the reflection ordering $\prec$ with $c$ means that we are always in the first situation ($w = u_i u_{i-1}$) if $m\geq 3$ and $u_1,\dots,u_m$ are indexed in increasing order.

Some properties of the $c$-compatible reflection ordering $\prec$ can also be reformulated as follows:

\begin{lemm}  \label{lemm:reforder}
Let $w\in NC_2$.  Then:
\begin{itemize}
    \item a $\prec$-decreasing factorization $w=v_1v_2$ where $v_1,v_2\in T$ is such that $\langle \rho(v_1)|\rho(v_2)\rangle\geq 0$.
    \item $w$ has a unique $\prec$-increasing factorization $w=v_1v_2$ where $v_1,v_2\in T$, and it is such that $\langle \rho(v_1)|\rho(v_2)\rangle\leq 0$.
\end{itemize}
\end{lemm}

By reinterpreting the definition of $\Delta^+$ in terms of the reflection ordering, the following is natural:

\begin{lemm} \label{clus_alterdef}
Assume that $\prec$ is a reflection ordering compatible with $c$.  Let $f=\{t_1,\dots,t_n\}\subset T$, indexed so that $t_1\succ \dots \succ t_n$.  Then $f$ is a positive $c$-cluster if and only if $c=t_1\cdots t_n$.
\end{lemm}

\begin{proof}
Assume $c=t_1\cdots t_n$ and $t_1\succ \dots \succ t_n$.  For $1\leq i<j\leq n$, we have $t_it_j \in NC_2$ by the subword property.  By the first point of Lemma~\ref{lemm:reforder}, $\langle \rho(t_i)  | \rho(t_j) \rangle \geq 0$.  It follows that $\{t_1,\dots, t_n\}$ is a $c$-cluster.

Reciprocally, assume $f=\{t_1,\dots,t_n\}$ is a $c$-cluster, indexed so that $c=t_1\cdots t_n$.  Suppose that there is an index $1\leq i < n $ such that $t_i\prec t_{i+1}$.  We get $\langle \rho(t_i)  | \rho(t_j) \rangle \leq 0$ from the second point of the Lemma~\ref{lemm:reforder}.  We also have $\langle \rho(t_i)  | \rho(t_{i+1}) \rangle \geq 0$ by definition of $\Delta^+$, so $\langle \rho(t_i)  | \rho(t_{i+1}) \rangle = 0$.  This means $t_it_{i+1} = t_{i+1}t_i$.  If we replace $t_it_{i+1}$ with $t_{i+1}t_i$ in $c=t_1\cdots t_n$, we get a factorization with is lexicographically bigger.  It means that after a certain number of such commutations, we arrive at a decreasing factorization.  This permits to conclude the proof.
\end{proof}

Now, it remains only to explain why $c$-compatible reflection orderings exist, for any standard Coxeter element $c$.  A well-known construction gives an answer in the bipartite case (see~\cite{ABW07} for details).  The general answer is given below (see Equation~\eqref{eq_sorting_reforder}), it follows from results in~\cite{CLS14} where the complexes $\Delta$ and $\Delta^+$ are shown to be {\it subword complexes}.

Consider the $S$-word obtained by $k$ repetitions of $s_1\cdots s_n$.  When $k$ is large enough, it contains subwords that are reduced word for the longest element $w_\circ$.  The lexicographically minimal such subword is called the $c$-{\it sorting word}.   It is thus a reduced word $ w_\circ = s_{i_1} s_{i_2} \cdots s_{i_{nh/2}}$.  By results of Dyer~\cite{Dye93}, there is a bijection between reduced words for $w_\circ$ and reflection orderings, and in the present case it gives the reflection ordering $\prec$ such that:
\begin{equation} \label{eq_sorting_reforder}
  s_{i_1} \prec s_{i_1}s_{i_2}s_{i_1} \prec s_{i_1}s_{i_2}s_{i_3}s_{i_2}s_{i_1} \prec \cdots.
\end{equation}
In particular, any reflection appear exactly once in the list above.  To see that $\prec$ is $c$-compatible, it remains only to reconcile our definition of $\Delta^+$ in Definition~\ref{def:clust1} with the definition from ~\cite{CLS14} in terms of subword complex and the $c$-sorting word.

%%%%%%%%%%%%%%%%%%%%%%%%
\section{Examples}
\label{sec:cluster_examples}
%%%%%%%%%%%%%%%%%%%%%%%%

\begin{exam} \label{exam:typeA3}
Consider the case $W=\mathfrak{S}_n$ and $c=(1,\ldots,n)$ as in Example~\ref{exam:typeA}. A $c$-compatible reflection ordering on $T=\{(i,j)\; : \; 1\leq i<j\leq n\}$ is given by the lexicographic order:
\[
  (i,j)\prec (k,l) \Longleftrightarrow 
  i<k \text{ or } (i=k \text{ and } j<l).
  %j>l \text{ or } (j=l \text{ and } i>k).
\]
In the case $n=4$, the $5$ $c$-clusters are represented in Figure~\ref{fig:clustersS4}.  Such a $c$-cluster $\{t_1 \succ t_2 \succ t_3\}$ is represented as follows: if $t_i=(j,k)$, we draw an arrow from $j$ to $k$ with label $i$.  Note that the face-counting polynomial from \eqref{eq:muNC} is given by $1+6q+10q^2+5q^3$.
% \begin{exam} \label{exam:typeA3}
% Consider the case $W=S_n$ and $c=(1,\ldots,n)$ as in Examples~\ref{exam:typeA} and  \ref{exam:typeA2}. A $c$-compatible reflection ordering on $T=(ij)$ is given by $(ij)\prec (kl)$ if $j>l$ or ($j=l$ and $i>k)$. One can represent a subset of transpositions by a graph on $\{1,\ldots,n\}$: in this representation, a $c$-cluster will be a \textit{noncrossing alternating tree}, that is a tree with noncrossing edges such that at each vertex $i$, either all edges are of the form $(k,i)$ or all are of the form $(i,j)$. Faces of $\Delta^+$ are then naturally encoded as noncrossing alternating forests (which also naturally occurred in \cite{Alb09}).
In general (other values of $n$), a $c$-cluster will be a \textit{noncrossing alternating tree}, that is a tree with noncrossing edges such that at each vertex $i$, neighbours are all $<i$ or all $>i$.  Faces of $\Delta^+$ are naturally identified with a natural notion of {\it noncrossing alternating forests} (which also naturally occurred in \cite{Alb09}).
\end{exam} 

\begin{figure}[th]
\centering
\includegraphics[width=\textwidth]{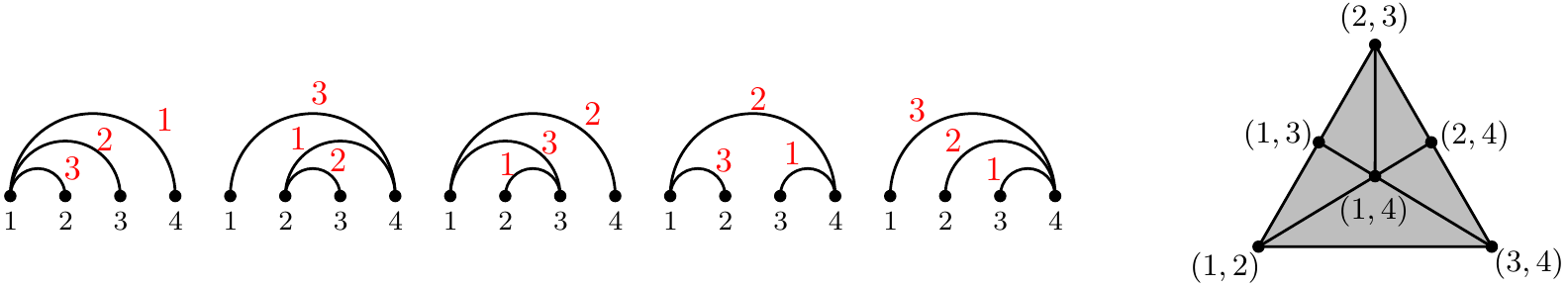}
\caption{Positive clusters for $W=\mathfrak{S}_4$ and $c=(1,2,3,4)$.\label{fig:clustersS4}}
\end{figure}

\begin{exam}
In the case of $\mathfrak{S}_4$ (type $A_3$) with the bipartite Coxeter element $c=(1,3,4,2)$, the 5 $c$-clusters are in Figure~\ref{fig:clustersS4bis}.   We use the $c$-compatible reflection ordering: $(2,3)\prec (1,3)\prec (2,4)\prec (1,4) \prec (3,4) \prec (1,2)$.  Note that that the two complexes in Figures~\ref{fig:clustersS4} and \ref{fig:clustersS4bis} are not isomorphic.
\end{exam}

\begin{figure}[th]
\centering
\includegraphics[width=\textwidth]{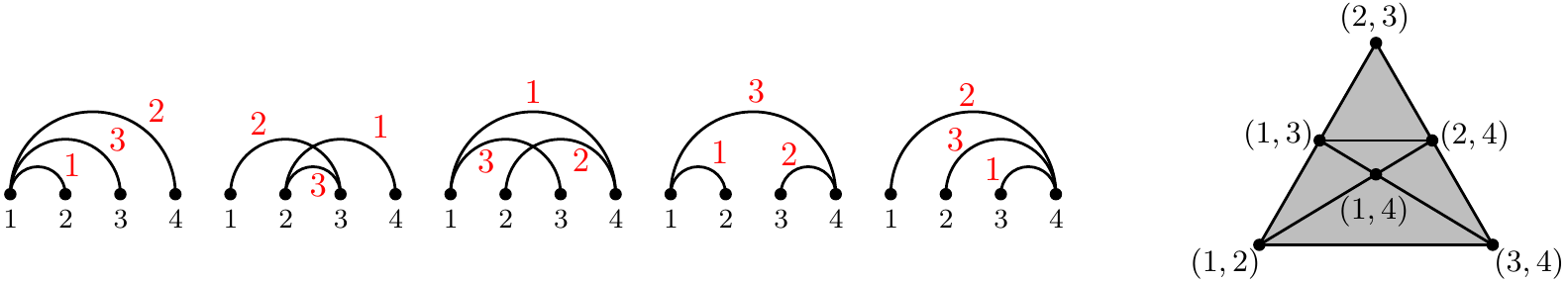}
\caption{Positive clusters for $W=\mathfrak{S}_4$ and $c=(1,3,4,2)$.\label{fig:clustersS4bis}}
\end{figure}

\begin{exam}
In type $B_3$, two representations of $\Delta^+$ for two different Coxeter elements are given in Figures~\ref{B3_1} and \ref{B3_2}. 
\end{exam}

\begin{figure}[th]
\centering
\includegraphics[scale=0.5]{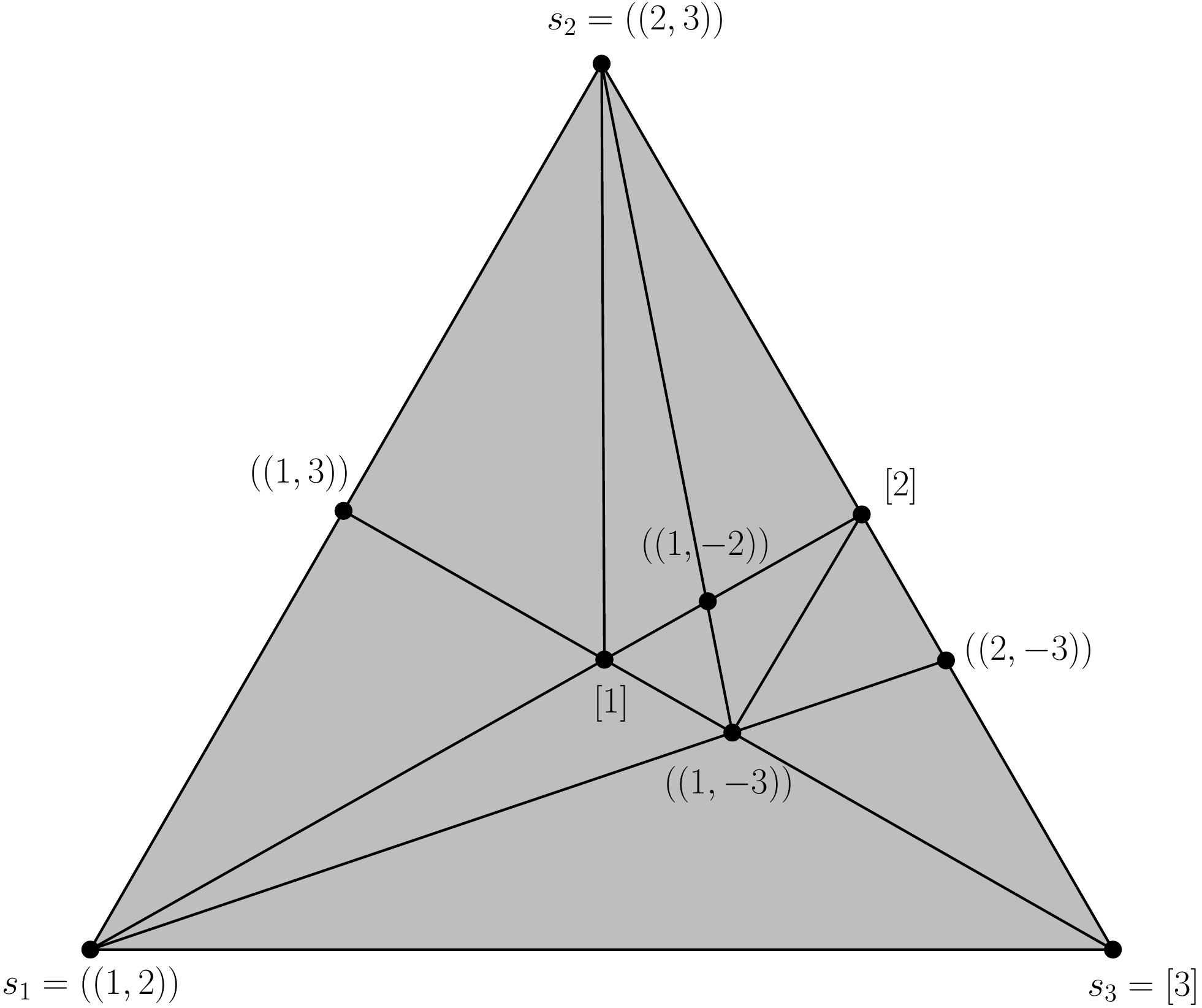}
\caption{Positive clusters for $W$ of type $B_3$ and $c=s_1s_2s_3=[[1,2,3]]$. A $c$-compatible order is   $((1,2))\prec((1,3))\prec[1]\prec((2,3))\prec((1,-2))\prec[2]\prec((1,-3))\prec((2,-3))\prec[3]$.\label{B3_1}}
\centering
\includegraphics[scale=0.5]{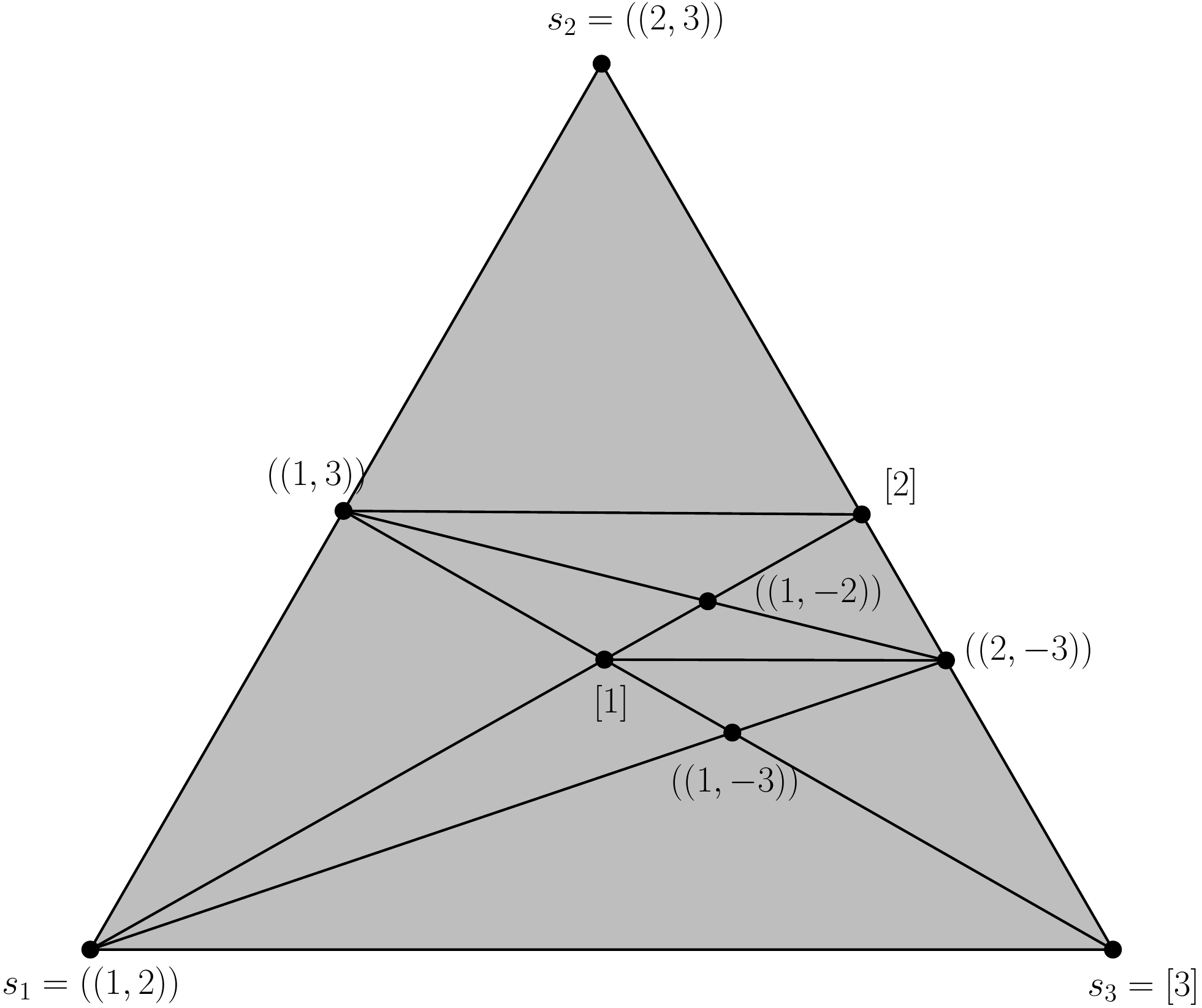}
\caption{Positive clusters for $W$ of type $B_3$ and $c=s_1s_3s_2=[[1,2,-3]]$. A $c$-compatible order is  $((1,2))\prec [3]\prec ((1,-3))\prec ((2,-3))\prec [1]\prec ((1,-2))\prec ((1,3))\prec [2]\prec ((2,3))$.
\label{B3_2}}
\end{figure}

\begin{figure}[th]
\centering
\includegraphics[width=0.8\textwidth]{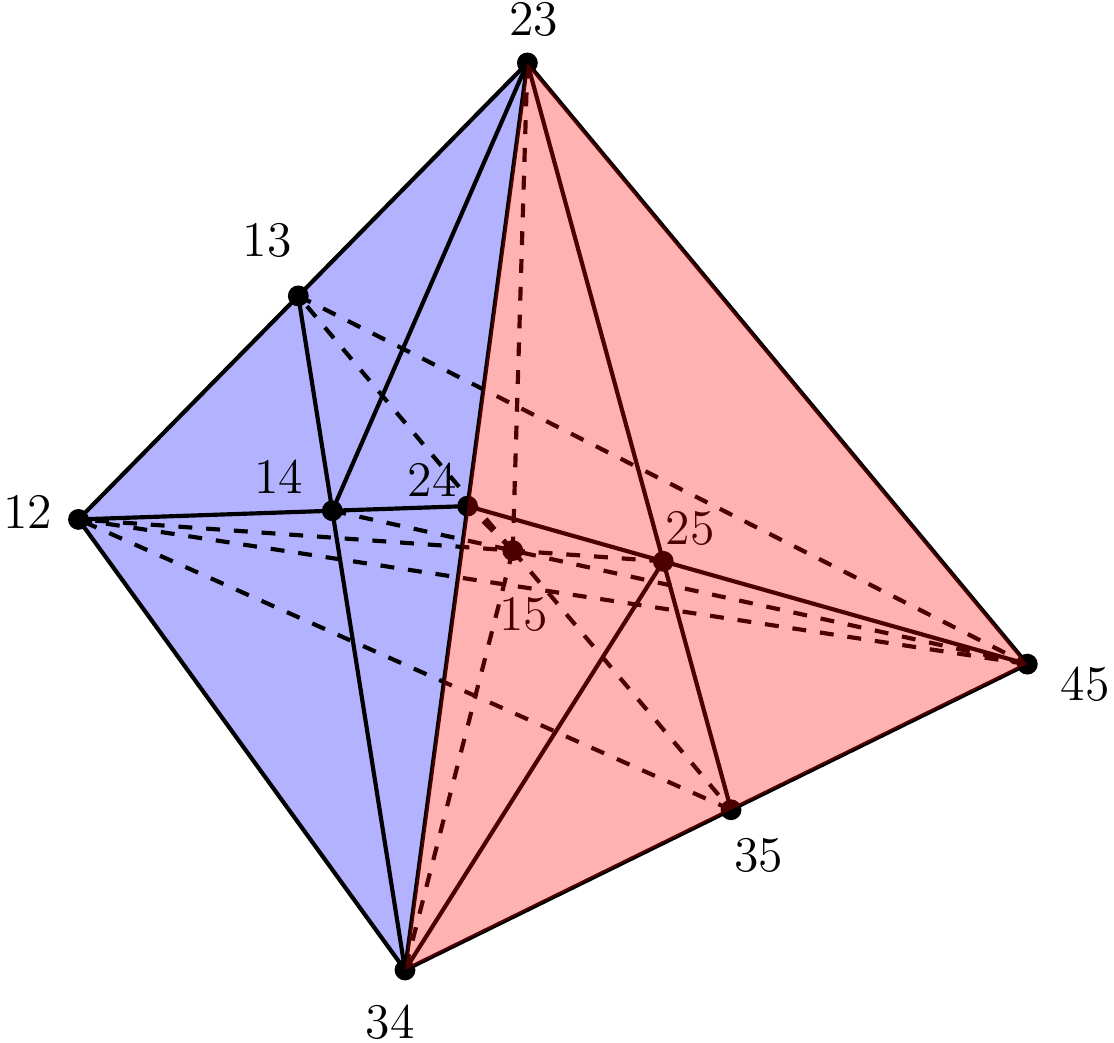}
\caption{Positive clusters for $W=\mathfrak{S}_5$ and $c=(1,2,3,4,5)$.}
\end{figure}

\begin{figure}[th]
\centering
\includegraphics[width=0.8\textwidth]{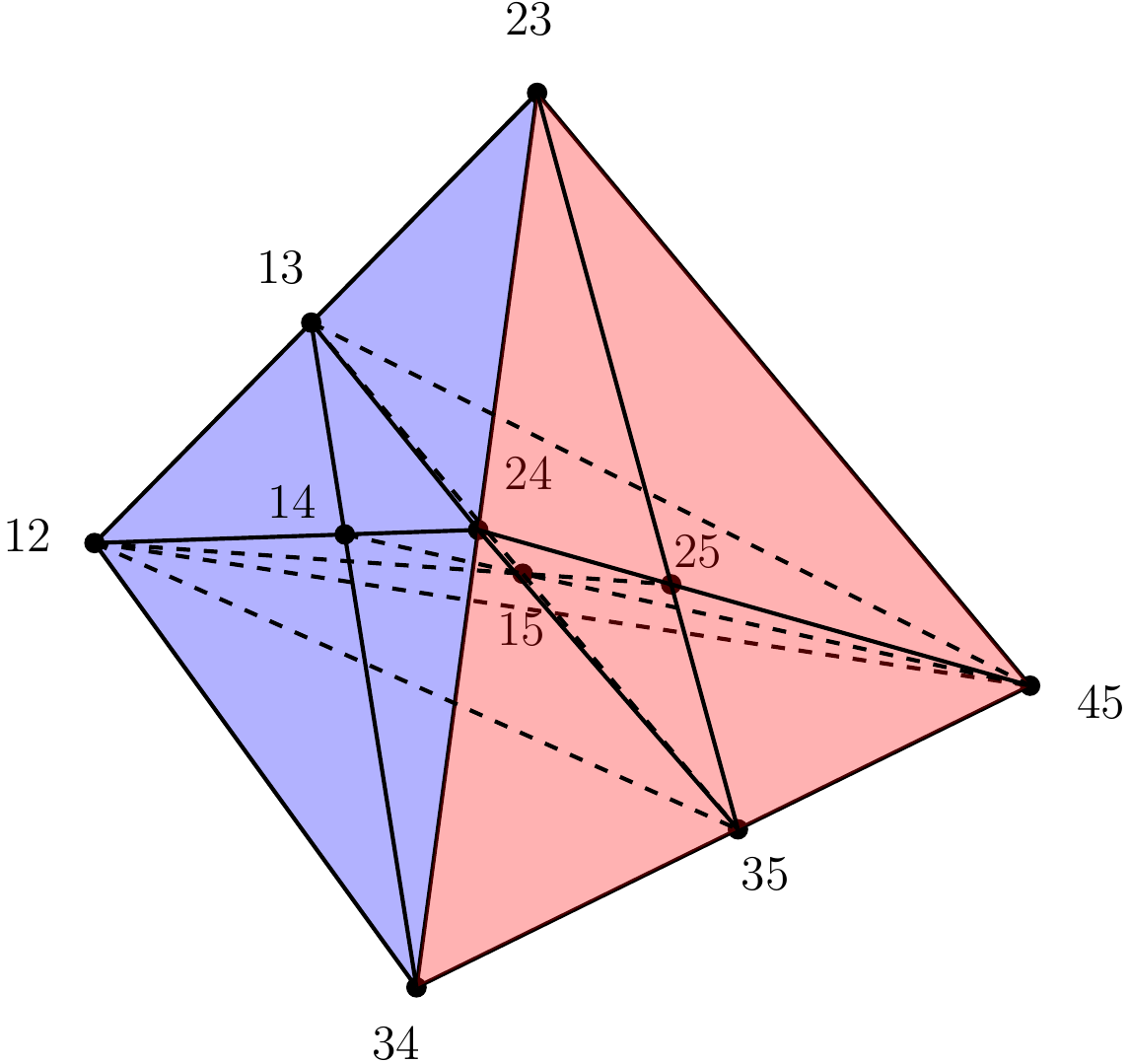}
\caption{Positive clusters for $W=\mathfrak{S}_5$ and $c=(1,3,5,4,2)$.}
\end{figure}


\begin{thebibliography}{10}

\bibitem{Alb09}
{\sc Albenque, M., and Nadeau, P.}
\newblock Growth function for a class of monoids.
\newblock In {\em 21st {I}nternational {C}onference on {F}ormal {P}ower
  {S}eries and {A}lgebraic {C}ombinatorics ({FPSAC} 2009)}, Discrete Math.
  Theor. Comput. Sci. Proc., AK. Assoc. Discrete Math. Theor. Comput. Sci.,
  Nancy, 2009, pp.~25--38.

\bibitem{Armstrong06}
{\sc Armstrong, D.}
\newblock Generalized noncrossing partitions and combinatorics of {C}oxeter
  groups.
\newblock {\em Mem. Amer. Math. Soc. 202}, 949 (2009), x+159.

\bibitem{AthaEnum}
{\sc Athanasiadis, C.~A.}
\newblock On some enumerative aspects of generalized associahedra.
\newblock {\em European J. Combin. 28}, 4 (2007), 1208--1215.

\bibitem{AthaBrad06}
{\sc Athanasiadis, C.~A., Brady, T., McCammond, J., and Watt, C.}
\newblock {$h$}-vectors of generalized associahedra and noncrossing partitions.
\newblock {\em Int. Math. Res. Not.\/} (2006), Art. ID 69705, 28.

\bibitem{ABW07}
{\sc Athanasiadis, C.~A., Brady, T., and Watt, C.}
\newblock Shellability of noncrossing partition lattices.
\newblock {\em Proc. Amer. Math. Soc. 135}, 4 (2007), 939--949.

\bibitem{Bes03}
{\sc Bessis, D.}
\newblock The dual braid monoid.
\newblock {\em Ann. Sci. \'Ecole Norm. Sup. (4) 36}, 5 (2003), 647--683.

\bibitem{Bes06}
{\sc Bessis, D.}
\newblock A dual braid monoid for the free group.
\newblock {\em J. Algebra 302}, 1 (2006), 55--69.

\bibitem{Bes15}
{\sc Bessis, D.}
\newblock Finite complex reflection arrangements are {$K(\pi,1)$}.
\newblock {\em Ann. of Math. (2) 181}, 3 (2015), 809--904.

\bibitem{Bia97}
{\sc Biane, P.}
\newblock Some properties of crossings and partitions.
\newblock {\em Discrete Math. 175}, 1-3 (1997), 41--53.

\bibitem{BJV18}
{\sc Biane, P., and Josuat-Verg\`es, M.}
\newblock Noncrossing partitions, {B}ruhat order and the cluster complex.
\newblock {\em Ann. Inst. Fourier (Grenoble) 69}, 5 (2019), 2241--2289.

\bibitem{Bir98}
{\sc Birman, J., Ko, K.~H., and Lee, S.~J.}
\newblock A new approach to the word and conjugacy problems in the braid
  groups.
\newblock {\em Adv. Math. 139}, 2 (1998), 322--353.

\bibitem{BWKP1}
{\sc Brady, T., and Watt, C.}
\newblock {$K(\pi,1)$}'s for {A}rtin groups of finite type.
\newblock In {\em Proceedings of the Conference on Geometric and Combinatorial
  Group Theory, Part I (Haifa, 2000)\/} (2002), vol.~94, pp.~225--250.

\bibitem{BWLattice}
{\sc Brady, T., and Watt, C.}
\newblock Non-crossing partition lattices in finite real reflection groups.
\newblock {\em Trans. Amer. Math. Soc. 360}, 4 (2008), 1983--2005.

\bibitem{BH93}
{\sc Bruns, W., and Herzog, J.}
\newblock {\em Cohen-{M}acaulay rings}, vol.~39 of {\em Cambridge Studies in
  Advanced Mathematics}.
\newblock Cambridge University Press, Cambridge, 1993.

\bibitem{Car72}
{\sc Carter, R.~W.}
\newblock Conjugacy classes in the {W}eyl group.
\newblock {\em Compositio Math. 25\/} (1972), 1--59.

\bibitem{Car69}
{\sc Cartier, P., and Foata, D.}
\newblock {\em Probl\`emes combinatoires de commutation et r\'earrangements}.
\newblock Lecture Notes in Mathematics, No. 85. Springer-Verlag, Berlin, 1969.

\bibitem{CLS14}
{\sc Ceballos, C., Labb\'{e}, J.-P., and Stump, C.}
\newblock Subword complexes, cluster complexes, and generalized
  multi-associahedra.
\newblock {\em J. Algebraic Combin. 39}, 1 (2014), 17--51.

\bibitem{Cha05}
{\sc Chapoton, F.}
\newblock Enumerative properties of generalized associahedra.
\newblock {\em S\'{e}m. Lothar. Combin. 51\/} (2004/05), Art. B51b, 16.

\bibitem{CFZ05}
{\sc Chapoton, F., Fomin, S., and Zelevinsky, A.}
\newblock Polytopal realizations of generalized associahedra.
\newblock {\em Canad. Math. Bull. 45}, 4 (2002), 537--566.
\newblock Dedicated to Robert V. Moody.

\bibitem{DDGKM15}
{\sc Dehornoy, P., Digne, F., Godelle, E., Krammer, D., and Michel, J.}
\newblock {\em Foundations of {G}arside theory}, vol.~22 of {\em EMS Tracts in
  Mathematics}.
\newblock European Mathematical Society (EMS), Z\"{u}rich, 2015.
\newblock Author name on title page: Daan Kramer.

\bibitem{Dye93}
{\sc Dyer, M.~J.}
\newblock Hecke algebras and shellings of {B}ruhat intervals.
\newblock {\em Compositio Math. 89}, 1 (1993), 91--115.

\bibitem{FZ03}
{\sc Fomin, S., and Zelevinsky, A.}
\newblock {$Y$}-systems and generalized associahedra.
\newblock {\em Ann. of Math. (2) 158}, 3 (2003), 977--1018.

\bibitem{Fro99}
{\sc Fr{\"o}berg, R.}
\newblock Koszul algebras.
\newblock In {\em Advances in commutative ring theory ({F}ez, 1997)}, vol.~205
  of {\em Lecture Notes in Pure and Appl. Math.} Dekker, New York, 1999,
  pp.~337--350.

\bibitem{Hat03}
{\sc Hatcher, A.}
\newblock {\em Algebraic topology}.
\newblock Cambridge University Press, Cambridge, 2002.

\bibitem{Humphreys90}
{\sc Humphreys, J.~E.}
\newblock {\em Reflection groups and {C}oxeter groups}, vol.~29 of {\em
  Cambridge Studies in Advanced Mathematics}.
\newblock Cambridge University Press, Cambridge, 1990.

\bibitem{IngallsThomas}
{\sc Ingalls, C., and Thomas, H.}
\newblock Noncrossing partitions and representations of quivers.
\newblock {\em Compos. Math. 145}, 6 (2009), 1533--1562.

\bibitem{IS17}
{\sc Ishibe, T., and Saito, K.}
\newblock Zero loci of skew-growth functions for dual {A}rtin monoids.
\newblock {\em J. Algebra 480\/} (2017), 1--21.

\bibitem{Koba90}
{\sc Kobayashi, Y.}
\newblock Partial commutation, homology, and the {M}\"obius inversion formula.
\newblock In {\em Words, languages and combinatorics ({K}yoto, 1990)}. World
  Sci. Publ., River Edge, NJ, 1992, pp.~288--298.

\bibitem{Koz08}
{\sc Kozlov, D.}
\newblock {\em Combinatorial algebraic topology}, vol.~21 of {\em Algorithms
  and Computation in Mathematics}.
\newblock Springer, Berlin, 2008.

\bibitem{Kra05}
{\sc Krammer, D.}
\newblock Braid groups.
\newblock {\em http://www.warwick.ac.uk/$\sim$ masbal/MA4F2Braids/braids.pdf\/}
  (2005).

\bibitem{UKra}
{\sc Krähmer, U.}
\newblock Notes on {K}oszul algebras.

\bibitem{LV12}
{\sc Loday, J.-L., and Vallette, B.}
\newblock {\em Algebraic operads}, vol.~346 of {\em Grundlehren der
  Mathematischen Wissenschaften [Fundamental Principles of Mathematical
  Sciences]}.
\newblock Springer, Heidelberg, 2012.

\bibitem{MAJ94}
{\sc Majid, S.}
\newblock Algebras and {H}opf algebras in braided categories.
\newblock In {\em Advances in {H}opf algebras ({C}hicago, {IL}, 1992)},
  vol.~158 of {\em Lecture Notes in Pure and Appl. Math.} Dekker, New York,
  1994, pp.~55--105.

\bibitem{McCammond06}
{\sc McCammond, J.}
\newblock Noncrossing partitions in surprising locations.
\newblock {\em Amer. Math. Monthly 113}, 7 (2006), 598--610.

\bibitem{McS17}
{\sc McCammond, J., and Sulway, R.}
\newblock Artin groups of {E}uclidean type.
\newblock {\em Invent. Math. 210}, 1 (2017), 231--282.

\bibitem{MS00}
{\sc Milinski, A., and Schneider, H.-J.}
\newblock Pointed indecomposable {H}opf algebras over {C}oxeter groups.
\newblock In {\em New trends in {H}opf algebra theory ({L}a {F}alda, 1999)},
  vol.~267 of {\em Contemp. Math.} Amer. Math. Soc., Providence, RI, 2000,
  pp.~215--236.

\bibitem{Muh15}
{\sc M\"{u}hle, H.}
\newblock {E}{L}-shellability and noncrossing partitions associated with
  well-generated complex reflection groups.
\newblock {\em European J. Combin. 43\/} (2015), 249--278.

\bibitem{OS80}
{\sc Orlik, P., and Solomon, L.}
\newblock Combinatorics and topology of complements of hyperplanes.
\newblock {\em Invent. Math. 56}, 2 (1980), 167--189.

\bibitem{PS19}
{\sc Paolini, G., and Salvetti, M.}
\newblock Proof of the ${K}(\pi,1)$ conjecture for affine {A}rtin groups, 2019.

\bibitem{PP05}
{\sc Polishchuk, A., and Positselski, L.}
\newblock {\em Quadratic algebras}, vol.~37 of {\em University Lecture Series}.
\newblock American Mathematical Society, Providence, RI, 2005.

\bibitem{Priddy}
{\sc Priddy, S.~B.}
\newblock Koszul resolutions.
\newblock {\em Trans. Amer. Math. Soc. 152\/} (1970), 39--60.

\bibitem{Rea07}
{\sc Reading, N.}
\newblock Clusters, {C}oxeter-sortable elements and noncrossing partitions.
\newblock {\em Trans. Amer. Math. Soc. 359}, 12 (2007), 5931--5958.

\bibitem{RS09}
{\sc Reading, N., and Speyer, D.~E.}
\newblock Cambrian fans.
\newblock {\em J. Eur. Math. Soc. (JEMS) 11}, 2 (2009), 407--447.

\bibitem{TAK05}
{\sc Takeuchi, M.}
\newblock A survey on {N}ichols algebras.
\newblock In {\em Algebraic structures and their representations}, vol.~376 of
  {\em Contemp. Math.} Amer. Math. Soc., Providence, RI, 2005, pp.~105--117.

\bibitem{Wac07}
{\sc Wachs, M.~L.}
\newblock Poset topology: tools and applications.
\newblock In {\em Geometric combinatorics}, vol.~13 of {\em IAS/Park City Math.
  Ser.} Amer. Math. Soc., Providence, RI, 2007, pp.~497--615.

\bibitem{Zhang20}
{\sc Zhang, Y.}
\newblock {\em Combinatorics of Milnor fibres of reflection arrangements}.
\newblock PhD thesis, University of Sidney, 2020.

\bibitem{Zie95}
{\sc Ziegler, G.~M.}
\newblock {\em Lectures on polytopes}, vol.~152 of {\em Graduate Texts in
  Mathematics}.
\newblock Springer-Verlag, New York, 1995.

\end{thebibliography}
\end{document}